\documentclass[11pt,reqno]{amsart}
\newcommand {\norm}[1] {\left\| #1 \right\|}
\newcommand{\vertiii}[1]{{\left\vert\kern-0.15ex\left\vert\kern-0.15ex\left\vert #1
		\right\vert\kern-0.15ex\right\vert\kern-0.15ex\right\vert}}
\usepackage{fullpage,geometry}
\usepackage{amssymb}
\usepackage[usenames, dvipsnames]{color}
\usepackage{hyperref}

\usepackage{verbatim}

\numberwithin{equation}{section}

\newtheorem{theorem}{{\bf Theorem}}[section]
\theoremstyle{definition} \newtheorem{definition}[theorem]{\bf Definition}

\theoremstyle{plain} \newtheorem{lemma}[theorem]{Lemma}

\usepackage{amssymb}
\usepackage{bbm}
\usepackage{anysize}
\marginsize{1in}{1in}{1in}{1in}

\DeclareMathOperator*{\esssup}{{ess\,sup}}

\newcommand{\be}{\begin{eqnarray*}}
	\newcommand{\en}{\end{eqnarray*}}
\newcommand{\bes}{\begin{eqnarray}}
\newcommand{\ens}{\end{eqnarray}}

\def\nn{\nonumber}

\def\bq{\begin{equation}}
\def\eq{\end{equation}}
\def\bqq{\begin{eqnarray*}}
	\def\eqq{\end{eqnarray*}}

\begin{document}
			\title[On existence and  regularity of a terminal]{ On existence and  regularity of a terminal value problem for the time fractional diffusion equation }         
		\author[N.H. Tuan]{Nguyen Huy Tuan}
		\address[N.H. Tuan]{Applied Analysis Research Group,
			Faculty of Mathematics and Statistics,  Ton Duc Thang University, Ho Chi Minh City, Vietnam}
		\email{nguyenhuytuan@tdt.edu.vn}
		
		\author[T.B. Ngoc]{Tran Bao Ngoc}
		\address{  Department of Mathematical Economics Banking University of Ho Chi Minh City Ho Chi Minh City Vietnam}%
		\email{tr.bao.ngoc@gmail.com }
		
		\author[Y. Zhou] {Yong Zhou}
		\address{  Faculty of Mathematics and Computational Science, Xiangtan University, Hunan 411105, China}%
		\email{yzhou@xtu.edu.cn  }
		
		\author[D. O'Regan] {Donal O'Regan}
		\address{  School of Mathematics, Statistics and Applied Mathematics, National University of
			Ireland, Galway, Ireland }%
		\email{donal.oregan@nuigalway.ie  }

			\maketitle

\begin{abstract}		
 In this paper we consider a final value problem for a diffusion equation with time-space fractional differentiation on a bounded domain $D$ of $ \mathbb{R}^{k}$, $k\ge 1$, which includes the fractional power $\mathcal L^\beta$, $0<\beta\le 1$, of a symmetric uniformly elliptic operator $\mathcal L$ defined on $L^2(D)$. A representation of solutions is given by using the Laplace transform and the spectrum of $\mathcal L^\beta$. We establish some existence and regularity results for our problem in both the  linear and nonlinear case.
	\end{abstract}
	\tableofcontents

\section{Introduction}
The nonlinear diffusion equations, an important class of parabolic equations, come from many diffuse phenomena that appear widely in nature. They are proposed as mathematical models of physical problems in many areas, such as filtering, phase transition, biochemistry and dynamics of biological groups. Many new ideas and methods have been developed to consider some various kinds of nonlinear diffusion equation. We can list some selected and impressive works in recent time, for example L. Caffarelli et al \cite{Ca},  F. Duzaar et al  \cite{Du,Du1,Du2}, J.L.  Vazquez et al \cite{Va,Te,Va3,Va4}  and the references therein.

We present existence and regularity estimates for
the solution to a final  boundary value problem for a { space-time} fractional diffusion
equation.  Let $D$ be an open and bounded domain in $\mathbb{R}^k, (k \ge 1)$  with boundary $\partial D$. Given $0<\alpha<1$ and $0 <\beta \le 1$, a forcing (or source) function $F$,  we consider the  final value problem for the time fractional diffusion equation
\begin{align}
& ^c{\hspace*{-0.05cm}}D_t^\alpha u(t,x) = -  \mathcal L^{\beta }  u(t,x) + F(t,x,u(t,x)),    \quad (t,x)\in J\times D ,    \label{mainpro1}
\end{align}
with the boundary condition
\begin{align}
\mathcal H  u(t,x) = 0, \quad (t,x) \in  J\times \partial
D , \label{mainpro2}
\end{align}
and the final condition
\begin{align}
u(x,T) = \varphi(x), \quad x \in   D, \label{mainpro3}
\end{align}
where $\varphi$ is a given function.  Here  $J$ is the interval $(0,T)$,    The notation $^c{\hspace*{-0.05cm}}D_t^\alpha$ for $0<\alpha<1$ represents the left Caputo fractional derivative of order $\alpha$ which is defined by
\begin{align}
^c{\hspace*{-0.05cm}}D_t^\alpha v(t):= I_t^{1-\alpha}D_tv(t), \quad t\ge 0, \nn
\end{align}
provided that $I^\alpha_tv(t):=g_\alpha(t)\star v(t)$, here $g_\alpha(t)=\frac{1}{\Gamma(\alpha)}t^{\alpha-1}$, $t>0$, $\star$ denotes the convolution. For $\alpha=1$, we consider the usual time derivative $\frac{\partial u}{\partial t}$.    The fractional power $\mathcal L^{\beta}$ $0< \beta\le 1$  of the Laplacian operator $\mathcal L$ on $D$ is defined by its spectrum.  The symmetric uniformly elliptic operator is defined on the space $L^2(D)$ by
\begin{align*}
\mathcal L u(x) =  - \sum_{i=1}^k \frac{\partial }{\partial x_i} \left( \sum_{j=1}^k \mathcal L _{ij}(x) \frac{\partial}{\partial x_j}u(x) \right) + b(x)u(x),
\end{align*} provided that $\mathcal L _{ij}\in C^1\left(\overline{\Omega}\right)$, $b\in C\left(\overline{\Omega}\right)$, $b(x)\ge 0$ for all $x\in \overline{\Omega}$, $\displaystyle \mathcal L _{ij}=\mathcal L _{ji}, 1\le i,j\le k$, and $\displaystyle \xi^T \left[\mathcal L _{ij}(x) \right] \xi \ge L_0 |\xi|^2$ for some $L_0>0$, $x\in \overline{\Omega}$, $\xi=(\xi_1,\xi_2,...,\xi_k)\in \mathbb{R}^k$.    The equation  (\ref{mainpro1}) is equipped
with $ \,
\mathcal Hv=v \, \textrm{ or } \, \mathcal Hv = \frac{\partial v}{\partial n} + \kappa v$, $\kappa >0$, $n$ is the outer normal vector of $\partial D$.

\vspace*{0.1 cm}

The time fractional reaction diffusion equation arises in describing   "memory"  occurring in physics such
as plasma turbulence \cite{4}.
It
was introduced by Nigmatullin \cite{Ni} to describe diffusion in media
with fractal geometry, which is a special type of porous media and is applied in the flow in highly heterogeneous aquifer \cite{Be} and
single-molecular protein dynamics \cite{Kou}.
In a physical model presented
in \cite{32}, the fractional diffusion corresponds to a diverging jump length variance in the
random walk, and a fractional time derivative arises when the characteristic waiting
time diverges.

\vspace*{0.1 cm}

If the final condition \eqref{mainpro3}  is replaced by the initial condition
\begin{align} \label{mainpro33}
u(x,0)= u_0(x),~~x  \in D
\end{align}
then {Problem \eqref{mainpro1}, \eqref{mainpro2}, \eqref{mainpro33} is} called a forward problem (or an initial value { problem)} for time-space  fractional diffusion equations; for applications of this type of  equation see  \cite{Gal} and for the abstract form of  \eqref{mainpro1}-\eqref{mainpro33}   see \cite{Clement}.  Carvaho et  al \cite{Andrade}   established a local theory of mild solutions for Problem \eqref{mainpro1}-\eqref{mainpro33}  where $\mathcal L^{\beta }$ is
a sectorial (nonpositive) operator.  	B.H. Guswanto \cite{Gu} studied the existence and uniqueness of a local
mild solution for a class of initial value problems for nonlinear fractional evolution equations and the study of  existence of initial value problems was considered  by M. Warma et al  \cite{Gal}. A significant number of papers has been devoted to extend properties holding in the standard setting to the fractional one (see for example \cite{Dong,Ga1,Kim,Li,Taylor}).

Numerical approximation for solutions for Problem \eqref{mainpro1}-\eqref{mainpro33} was studied by B. Jin et al \cite{jin,jin2} and for other works on
fractional diffusion see  \cite{Liu,Zhou,Nochetto,Salgado}.
However, the literature on regularity of the initial value problem  for fractional diffusion-wave equations is  scarce;  for the linear case see \cite{Mclean2,Dang,Yamamoto1,
	Nochetto}, and for the nonlinear case see \cite{Gal,warma,Mu,Yamamoto2}. Although there are many works on direct problem, but the results on  inverse problem for fractional diffusion are scarce. We can list some papers of M. Yamamoto and his group see \cite{Ya5,Ya6,Ya7,Ya8,Ya9,Ya10}, of B. Kaltenbacher et al \cite{Bar,Bar1}  , of W. Rundell et al \cite{Run,Run1},  of J. Janno see \cite{Ja1,Ja2}, etc.

\vspace*{0.1 cm}

In practice, initial data of some problems may not be known since many phenomena cannot be measured at the initial time.  Phenomena can be observed at a final time $t=T$, such as, in the image processing area. A picture is not processed at the capturing time $t=0$. Instead, one wishes to recover the original information of the picture from its blurry form. Hence,  inverse problems or  terminal value problems or final value problems (IPs/FVPs), i.e., the fractional differential equations (FDEs) equipped with  final value data, have been considered. IPs/FVPs  are  important in engineering in detecting the previous status of physical fields from its present information.
If $F=0$ in \eqref{mainpro1},   Yamamoto et al \cite{Yamamoto1} showed that Problem  \eqref{mainpro1}-\eqref{mainpro3}  has a unique weak solution when  $\varphi \in H^2(\Omega)$.  If $b=0$ and $F(u(x,t))=F(x,t)$, Tuan et al \cite{Tuan} showed that Problem  \eqref{mainpro1}-\eqref{mainpro3}  has a unique weak solution when  $\varphi \in H^2(\Omega)$ and $F \in L^\infty (0,T; H^2(\Omega))$,  and other works on the homogeneous case for
Problem  \eqref{mainpro1}-\eqref{mainpro3} can be found in \cite{Liu,Wei,Jia,Tuan}.  When $\alpha=1$,  systems \eqref{mainpro1}-\eqref{mainpro3} are reduced to the backward problem for classical
reaction diffusion equations, and were studied in \cite{Au,Hao,Tuan1}.

To the best of the authors' knowledge this is the first paper that analyzes problem
(\ref{mainpro1})-(\ref{mainpro3}). We present existence and uniqueness results and
derive regularity estimates both in time and space.  In what follows, we analyze the difficulties of this problem.   By letting $u(t,.)=\mathcal O(t)(f,\varphi)$, the solution operator $\mathcal O(0)$ is really not bounded in $L^2(D)$. Hence continuity of mild solutions does not hold at the initial time $t=0$. In addition, since the fractional derivative is non-locally defined, if we put $v(t)=u(T-t)$ then  $\left.^c{\hspace*{-0.05cm}}D_s^\alpha u(s)\right|_{s=T-t}$ does not equal $^c{\hspace*{-0.05cm}}D_t^\alpha v(t)$, so the problem cannot be changed to an initial value problem.  As a result  we need new techniques to deal with the FVP (\ref{mainpro1})-(\ref{mainpro3}). To the best of our knowledge, the work on the final value problem is still limited.

Our  main results in this paper can be split into  two parts,   linear and nonlinear  source functions.
Linear models are sometimes good approximations of the real problems under consideration and provides  mathematical tools needed to study nonlinear phenomena,
especially for semi-linear and quasi-linear equations.  In part 1, we consider the regularity property of the solution in the linear case $F$.
We seek to address
the following question: If the data is regular, how regular is the solution?
Our task in this part is to  find  a suitable Banach space for the given data $(\varphi,F)$ in order to obtain  regularity results for the corresponding solution.
In part 2, we  discuss existence, uniqueness and regularity for the solutions to (\ref{mainpro1})-(\ref{mainpro3}) for the  nonlinear problem.  Our main motivation for deriving regularity results is that one needs it for a  rigorous study of a numerical scheme to approximate the solution.   To the best of our knowledge, regularity results on inverse initial value problems (final value problems) for fractional diffusion  is still unavailable in the literature. For initial value problems, W. McLean et al \cite{Mclean2}, B. Jin et al \cite{jin} and B. Ahmad et al \cite{Mu}  considered existence and regularity results of the solution in $C([0,T]; L^2(D))$. However it seems  the techniques in \cite{jin,Mu} cannot be applied for our problems (it is impossible to apply  some well-known  fixed point theorems with some  spaces  in  \cite{jin} for establishing unique solutions).
To overcome this we need data $\varphi$ in a suitable space and we will use a Picard iteration argument  and  then  develop some  new techniques  to obtain existence and regularity  of the solution.

The rest of this paper is organized as follows.  In section 2, we give  basic notations and preliminaries, and we propose a mild solution of our problem. In section 3, we give some regularity results of the linear inhomogeneous problem. Section 4 is devoted to existence and regularity for nonlinear problems.
\section{Notations and preliminaries}
\subsection{Functional space}
In this subsection, we introduce some functional spaces for solutions of FVP (\ref{mainpro1})-(\ref{mainpro3}). By $\{m_j\}_{j\ge 1}$ and $\{e_j(x)\}_{j\ge 1}$, we denote the spectrum and sequence of eigenfunctions of $\mathcal L$ which satisfy  $e_j\in \{v\in H^2(D): \mathcal Hv=0\}$, $\mathcal L e_j(x)=m_je_j(x)$, $0<m_1\le m_2\le ...\le m_j\le ...$, and $\lim\limits_{j\to \infty} m_j=\infty$. The sequence $\{e_j(x)\}_{j\ge 1}$ forms an orthonormal basic of the space $L^2(D)$.
For a given real number $p \ge  0$, the Hilbert scale space $H^{2p}(D)$ is defined by
\begin{align}
\left\{ v\in L^2(D) \quad \textrm{such that} \quad \norm{v}_{H^{2p}(D)}^2:= \sum_{j=1}^\infty (v,e_j)^2m_j^{2p}<\infty \right\}, \nn
\end{align}
where $(.,.)$ is the usual inner product of $L^2(D)$.
The fractional power $\mathcal L^{\beta}$, $\beta \ge 0$,  of the Laplacian operator $\mathcal L$ on $D$ is defined by
\begin{align}
\mathcal L^\beta v(x):= \sum_{j=1}^\infty (v,e_j) m_j^{\beta}e_j(x). \label{Lbeta}
\end{align}
Then, $\{m_j^\beta\}_{j\ge 1}$ is the spectrum of the operator $\mathcal L^\beta$. We denote by ${\bf V}_\beta$ the domain of $\mathcal L^\beta$, and then 
$${\bf V}_\beta=\{v\in L^2(D): \norm{\mathcal L^\beta v}<\infty\}$$ 
where $\norm{.}$ is the usual norm of $L^2(D)$, and ${\bf V}_\beta$ is a Banach space with respect to the norm $\norm{v}_{{\bf V}_\beta}=\norm{\mathcal L^\beta v}$. Moreover, the inclusion ${\bf V}_\beta \subset H^{2\beta}(D)$ holds for $\beta >0$. We identity the dual space $\left(L^2(D)\right)'=L^2(D)$ and define the domain ${\bf V}_{-\beta}:=D(\mathcal L^{-\beta})$ by the dual space of ${\bf V}_\beta$, i.e., ${\bf V}_{-\beta} = \left({\bf V}_{\beta}\right)'$. Then, ${\bf V}_{-\beta}$  is a Hilbert space endowed with the norm $$\norm{v}_{{\bf V}_{-\beta}} :=\left\{ \sum_{j=1}^\infty (v,e_j)_{-\beta,\beta}^2 m_j^{-2\beta} \right\}^{1/2}, $$ where $(.,.)_{-\beta,\beta}$ denotes the dual inner product between ${\bf V}_{-\beta}$ and ${\bf V}_{\beta}$. We note that the Sobolev embedding
${\bf V}_{\beta} \hookrightarrow L^2(D) \hookrightarrow {\bf V}_{-\beta}$
holds  for  $0<\beta < 1$, and $
(\tilde{v},v)_{-\beta,\beta}=(\tilde{v},v)$, for $\tilde{v}\in L^2(D)$, $v\in {\bf V}_{\beta}$. Hence,  we have
\begin{align}
(e_i,e_j)_{-\beta,\beta}= (e_i,e_j)=\delta_{ij} \label{ggIP}
\end{align}
where $\delta_{ij}$ is the Kronecker delta for $i,j\in \mathbb{N}$, $i,j\ge 1$.          Moreover, for given $p_1\ge 1$ and $0<\eta<1$, we denote by $ \mathcal X _{p_1,\eta}(J\times D)$ the set of all functions $f$ from $J$ to $L^{p_1}(D)$ such that
\begin{align}
\vertiii{f}_{\mathcal X _{p_1,\eta}}:=\esssup_{0\le t\le T}\int_0^t \norm{f(\tau,.)}_{p_1} (t-\tau)^{\eta-1} d\tau <\infty,\label{Fset}
\end{align}
where $\norm{.}_{p_1}$ is the norm of $L^{p_1}(D)$. Note that, for fixed $t>0$, the H\"older's inequality shows that
\begin{align}
\int_0^t \norm{f(\tau,.)}_{p_1} (t-\tau)^{\eta-1} d\tau \le \left[\int_0^t \norm{f(\tau,.)}_{p_1}^{p_2} d\tau \right]^{\frac{1}{p_2}} \left[\int_0^t (t-\tau)^{\frac{p_2(\eta-1)}{p_2-1}}  d\tau \right]^{\frac{p_2-1}{p_2}}. \nn
\end{align}
In the above inequality, we note that the function $\tau\to   (t-\tau)^{\frac{p_2(\eta-1)}{p_2-1}}$ is integrable  for $p_2>\frac{1}{\eta}$.   Therefore, if we let $L^{p_2}(0,T;L^{p_1}(D))$, $p_1,p_2\ge 1$,  be the space of all Bochner's measurable functions $f$ from $J$ to $L^{p_1}(D)$ such that 
\[
\norm{f}_{L^{p_2}(0,T;L^{p_1}(D))}:=\left[\int_0^t \norm{f(\tau,.)}_{p_1}^{p_2} d\tau \right]^{\frac{1}{p_2}} < \infty,
\]
 then the following inclusion holds
\begin{align}
L^{p_2}(0,T;L^{p_1}(D)) \subset  \mathcal X _{p_1,\eta}(J\times D), \quad \textrm{for } p_2>\frac{1}{\eta}, \label{inclusion}
\end{align}
and there exists a positive constant $C>0$ such that
\begin{align}
\vertiii{f}_{\mathcal X _{p_1,\eta}} \le C \norm{f}_{L^{p_2}(0,T;L^{p_1}(D))}, \label{Holder}
\end{align}
here, $C$ depends only on $p_2$, $\eta$, and $T$. Moreover, for a given number $s$ such that $0<s<\eta$,    we have $\mathcal X _{p_1,\eta-s}(J\times D) \subset \mathcal X _{p_1,\eta}(J\times D)$ since $\vertiii{f}_{\mathcal X _{p_1,\eta}} \le T^s \vertiii{f}_{\mathcal X _{p_1,\eta-s}}$.  Let $B$ be a Banach space, and we denote by $C([0,T],B)$ the space of all continuous functions from $[0,T]$ to $B$ endowed with the norm $\norm{v}_{C([0,T];B)} := \sup_{0\le t\le T} \norm{v(t)}_B$, and by $C^\theta([0,T],B)$ the subspace of $C([0,T];B)$ which includes all H\"older-continuous functions, and is equipped with the norm
\begin{align}
\vertiii{v}_{C^\theta([0,T],B)} :=   \sup_{0\le t_1<t_2\le T} \frac{\norm{v(t_2)-v(t_1)}_B}{|t_2-t_1|^\theta}. \nn
\end{align}
In some cases, a given function might  not be continuous at $t=0$. Hence, it is useful to consider the set $C((0,T];B)$  which consists of all continuous functions  from $(0,T]$ to $B$. We define by $C^{\rho}((0,T];B)$ the Banach space of all functions $v$ in $C((0,T];B)$ such that
\begin{align}
\norm{v}_{C^{\rho}((0,T];B)} := \sup_{0<t\le T} t^\rho \norm{v(t)}_B <\infty. \nn
\end{align}

Now, we discuss  solutions of the FVP for the fractional ordinary equation
\begin{align}
^c{\hspace*{-0.05cm}}D_t^\alpha v(t)  = g(t,v(t)) -  m  v(t), \, t\in J,  \quad \textrm{and} \quad v(T)=v_T , \label{FOE}
\end{align}
where $m$, $v_T$ are given real numbers. Here, we wish to find a representation formula for $v$ in terms of the given function $g$ and the final value data $v_T$. By writing $^c{\hspace*{-0.05cm}}D_t^\alpha=I_t^{1-\alpha}D_t$, and applying the fractional integral $I^\alpha_t$ on both sides of Equation (\ref{FOE}), { we obtain   $$v(t)=v(0)+ I^{\alpha}_t\left[g(t,v(t))-m v(t)\right].$$}
The Laplace transform yields that
$$
\widehat{v}=\frac{\lambda^{\alpha-1}}{\lambda^\alpha+m} v(0)  + \frac{1}{\lambda^\alpha +m} \widehat{g}(v), $$
where $\widehat{v}$ is the Laplace transform of $v$. Hence, the inverse Laplace transform implies
\begin{align}
v(t)= v(0) E_{\alpha,1}(-mt^\alpha)+   g(t,v(t))\star \left[t^{\alpha-1}E_{\alpha,\alpha}(-mt^\alpha)\right] . \label{init}
\end{align}
Here, $E_{\alpha,1}$ and $E_{\alpha,\alpha}$ are the Mittag-Leffler functions
{ which are generally defined by} $  E_{a,b}(z)=\sum_{k=1}^{\infty} \frac{z^k}{\Gamma(a k +b)}$, for $a>0$, $b\in \mathbb{R}$,  $z\in \mathbb{C}$. Now, a representation of the solution of FVP (\ref{FOE}) can be obtained by substituting $t=T$ into (\ref{init}), and using the final value data $v(T)=v_T$, i.e.,
\begin{align}
v(t)=  g(t,v(t))\star \widetilde{E}_{\alpha,\alpha}(-mt^\alpha) + \Big[ v_T - g(T,v(T))\star \widetilde{E}_{\alpha,\alpha}(-m T^\alpha)  \Big] \frac{E_{\alpha,1}(-mt^\alpha)}{E_{\alpha,1}(-mT^\alpha)} , \nn
\end{align}
where  $$\widetilde{E}_{\alpha,\alpha}(-mt^\alpha):=t^{\alpha-1} E_{\alpha,\alpha}(-mt^\alpha),$$ and
\begin{align}
g(T,v(T))\star \widetilde{E}_{\alpha,\alpha}(-m T^\alpha)  := \left.g(r,v(r))\star \widetilde{E}_{\alpha,\alpha}(-m r^\alpha)\right|_{r=T}. \label{ConvolutionDenote}
\end{align}

\subsection{Mild solutions of FVP (\ref{mainpro1})-(\ref{mainpro3}) and unboundedness of solution operators}
A representation of solutions and the definition of mild solutions are given in this subsection, and then we analyze the unboundedness of solution operators. By the definition (\ref{Lbeta}) of $\mathcal L^\beta$, the identity $\mathcal L^\beta e_j(x)= m_j^\beta e_j(x)$ holds. Hence, in view of the Fourier expansion $u(t,x)=\sum_{j=1}^\infty u_j(t)e_j(x)$, where  $u_j(t)=(u(t,.),e_j)$, Equation (\ref{mainpro1}) can be rewritten as
\begin{align}
& \left(^c{\hspace*{-0.05cm}}D_t^\alpha \sum_{j=1}^\infty u_j(t)e_j ,e_j\right) = - \left( \mathcal L^{\beta }  \sum_{j=1}^\infty u_j(t)e_j,e_j \right) + \left(F(t,x,u(t,x)),e_j\right),    \quad t\in J.\nn
\end{align}
This is equivalent to the equation $$^c{\hspace*{-0.05cm}}D_t^\alpha   u_j(t)  = F_j(t,u(t)) - m_j^\beta u_j(t),~~ \quad F_j(t,u(t))=(F(t,x,u(t,x)),e_j).$$ 
By applying the method of solutions of FVPs for fractional ordinary equations in Subsection 2.1, and using the final value data (\ref{mainpro3}), we derive
{ \begin{align}
	\,& \hspace*{-1.3cm} u_j(t)= F_j(t,u(t))\star \widetilde{E}_{\alpha,\alpha}(-m_j^\beta t^\alpha)  \label{uj}\\
	\,& \hspace*{-0.5cm}   + \Big[\varphi_j - F_j(T,u(T))\star \widetilde{E}_{\alpha,\alpha}(-m_j^\beta T^\alpha)  \Big] \frac{E_{\alpha,1}(-m_j^\beta t^\alpha)}{E_{\alpha,1}(-m_j^\beta T^\alpha)}  ,\nn
	\end{align}}
where $\varphi_j =(\varphi,e_j)$. Therefore, we obtain a spectral representation for $u$ as follows:
{ \begin{align}
	u(t,x)  =\,& \sum_{j=1}^\infty F_j(t,u(t))\star \widetilde{E}_{\alpha,\alpha}(-m_j^\beta t^\alpha)e_j(x) \nn \\
	+\,& \sum_{j=1}^\infty \Big[\varphi_j - F_j(T,u(T))\star \widetilde{E}_{\alpha,\alpha}(-m_j^\beta T^\alpha)  \Big] \frac{E_{\alpha,1}(-m_j^\beta t^\alpha)}{E_{\alpha,1}(-m_j^\beta T^\alpha)}e_j(x)  .\nn
	\end{align}}
For $g:(0,T) \to L^2(D)$ and $v\in L^2(D)$, let us denote by $\mathcal O_{n}(t,x)$, $1\le n\le3$, the following operators
\begin{align}
\mathcal O_{1}(t,x)g := \,& \sum_{j=1}^\infty g_j(t) \star \widetilde{E}_{\alpha,\alpha}(-m_j^\beta t^\alpha)e_j(x), ~
\mathcal O_{2} (t,x) v  :=  \sum_{j=1}^\infty v_j  \frac{E_{\alpha,1}(-m_j^\beta t^\alpha)}{E_{\alpha,1}(-m_j^\beta T^\alpha)}e_j(x),
\nn
\end{align}
and $\mathcal O_3(t,x)=-\mathcal O_2(t,x)\mathcal O_1(T,x)$, for $(t,x)\in J\times D$. Then, the solution $u$ can be represented as
\begin{align}
u(t,x)= \mathcal O_1 (t,x) F   + \mathcal O_2(t,x)   \varphi  + \mathcal O_3(t,x) F ,\label{Nmild}
\end{align}
where we understand $F(t,u)=F(t,.,u(t,.))$ is a function of $x$ for fixed $t$.

\vspace*{0.2cm}

One of the most important things, when we consider the well-posedness of a PDE, is the boundedness of solution operators. { Corresponding to the initial value problem \eqref{mainpro1}, \eqref{mainpro2}, \eqref{mainpro33}, the solution operators are usually bounded in $L^2(D)$; see e.g.,  \cite{Mclean2,Mu,Salgado,Yamamoto2,Kian}.   Unfortunately, some solution operators of FVP (\ref{mainpro1})-(\ref{mainpro3}) are not bounded on $L^2(D)$ at $t=0$.   For this purpose, we recall that, for $0<\alpha<1$ and $z<0$, there exist  positive constants $c_\alpha$, $\widehat{c}_\alpha$ such that
	\begin{align}
	\frac{c_\alpha}{1+|z|} \le |E_{\alpha ,1}(z)| \le  \frac{\widehat{c}_\alpha}{1+|z|}, \quad
	|E_{\alpha ,\alpha}(z)| \le \min\left\{ \frac{\widehat{c}_\alpha}{1+|z|}, \frac{\widehat{c}_\alpha}{1+|z|^2} \right\} , \label{BasicMLInequality}
	\end{align}
	see, for example \cite{Samko, Podlubny, Diethelm}. Now, let $v_0$ be defined by $v_{0,j}:=(v_0,e_j)=j^{-1/2}m_j^{-\beta}$, $j\ge 1$. Then, it is easy to see that $v_0$ belongs to ${\bf V}_{\beta\gamma}$ for $0\le \gamma <1$, and does not for $\gamma \ge 1$. Using the inequalities (\ref{BasicMLInequality}), we have
	\begin{align}
	\norm{\mathcal O_2(0,.) v_0}^2=\,&   \sum_{j=1}^\infty   \frac{v^2_{0,j}}{E^2_{\alpha,1}(-m_j^\beta T^\alpha)}    \ge c_\alpha^{-2}  \sum_{j=1}^\infty    v^2_{0,j} (1+m_j^\beta T^\alpha)^2   \nn\\
	\ge\,& c_\alpha^{-2} T^{2\alpha}   \sum_{j=1}^\infty    v^2_{0,j}  m_j^{2\beta}   \hspace*{-0.05cm} =   c_\alpha^{-2} T^{2\alpha}   \sum_{j=1}^\infty    j^{-1}  = \infty, \nn
	\end{align}
	which shows the unboundedness of $\mathcal O_2(0,.)$ on $L^2(D)$. Similarly, the unboundedness of $\mathcal O_3(0,.)$ on $L^2(D)$ can be shown.}

\section{FVP with a linear source}


\noindent In this section, we study the regularity of mild solutions of FVP (\ref{mainpro1})-(\ref{mainpro3}) corresponding to the linear  source function $F$, i.e., $F(t,x,u(t,x))=F(t,x)$ which does not include $u$. We will investigate the regularity of the following FVP
\begin{equation}
\begin{cases}
\begin{array}{llll}
{}^c{\hspace*{-0.05cm}}D_t^\alpha u(t,x)  &=\,\, -  \mathcal L^{\beta }  u(t,x) + F(t,x), & & (t,x)\in J\times D,\\
\mathcal H  u(t,x) &= \,\, 0, & & (t,x) \in  J\times \partial
D,\\
u(x,T)&= \,\, \varphi(x), & & x \in   D,\\
\end{array}
\end{cases}\label{LINEAR}
\end{equation}
where $\varphi$, $F$ will be specified later. In order to consider this problem, it it necessary to give a definition of mild solutions  based on (\ref{Nmild}) as follows.
\begin{definition} If a function $u$ belongs to $L^p(0,T;L^q(D))$, for some $p,q\ge 1$, and satisfies the equation
	\begin{align}
	u(t,x)= \mathcal O_1 (t,x) F    + \mathcal O_2(t,x)   \varphi  + \mathcal O_3(t,x) F  ,\label{mild}
	\end{align}
	then $u$ is said to be a mild solution of FVP  (\ref{LINEAR}).
\end{definition}

In what follows, we introduce some assumptions on the final value data $\varphi$ and the linear source function $F$.
\vspace*{0.2cm}
\begin{itemize}
	\item (R1) $0<{p}, {q}<1$ such that ${p}+{q}=1$; \vspace*{0.1cm}
	\item (R2) $\displaystyle 0<r\le \frac{1-\alpha q}{\alpha q}$; \vspace*{0.2cm}
	\item (R3) $0<s<\min \big(\alpha q,1-\alpha q \big)$; \vspace*{0.1cm}
	\item (R4) $
	\displaystyle 0 < p' \le   p- \frac{s}{\alpha} , \hspace*{1.32cm} q' =1- p' , \quad 0<r\le \frac{1-\alpha q'}{\alpha q'}
	$; \vspace*{0.1cm}
	\item (R5) $\displaystyle 0 \le \widehat{q} \le  \min\left(p,q,\frac{s}{\alpha}\right), \quad \widehat{p}=1-\widehat{q}, \hspace*{0.45cm}  0<\widehat{r} \le  \frac{1-\alpha}{\alpha }$. \vspace*{0.1cm}
\end{itemize}

In the following lemma, we will show that solutions of FVP (\ref{LINEAR}) must be bounded by a power function $t^{-\alpha q}$, for some appropriate number $q$, i.e.,
\begin{align}
\norm{u(t,.)} \lesssim t^{-\alpha q}, \nn
\end{align}
for all $0<t\le T$.

\begin{lemma} \label{Lemma1} Let $p,q$ be defined by (R1), and $u$ satisfies (\ref{mild}). If $\varphi\in {\bf V}_{\beta{p}}$,  and $F\in \mathcal X _{2,\alpha{q}}(J\times D)$, then there exists a constant $C_0>0$ such that
	\begin{align}
	\norm{u(t,.)} \le C_0\left(\norm{\varphi}_{{\bf V}_{\beta{p}}} + \vertiii{F}_{\mathcal X _{2,\alpha{q}}}\right)t^{-\alpha{q}} .     \label{RLemma1}
	\end{align}
\end{lemma}

\begin{proof}   The inequalities  (\ref{BasicMLInequality}) shows that
	\begin{align}
	\,& E_{\alpha,\alpha}(-m_j^\beta (t-\tau)^\alpha) \le  \widehat{c}_\alpha  \big[1+m_j^\beta (t-\tau)^\alpha\big]^{- {p}}\le  \widehat{c}_\alpha m_j^{- \beta{p}} (t-\tau)^{- \alpha{p}}.
	\end{align}	
	Combined with the definition  $\displaystyle  \mathcal O_1 (t,x) F  =   \sum_{j=1}^\infty F_j(t) \star \widetilde{E}_{\alpha,\alpha}(-m_j^\beta t^\alpha)e_j(x) $,
	we have that
	\begin{align}
	\norm{\mathcal O_1 (t,.) F}
	\le\,& \int_0^t \norm{\sum_{j=1}^\infty F_j(\tau) \widetilde{E}_{\alpha,\alpha}(-m_j^\beta (t-\tau)^\alpha)e_j} d\tau  \nn \\
	=\,& \int_0^t \left\{\sum_{j=1}^\infty F_j^2(\tau) E^2_{\alpha,\alpha}(-m_j^\beta (t-\tau)^\alpha) (t-\tau)^{2\alpha-2} \right\}^{1/2} d\tau \nn\\
	\le\,& \widehat{c}_\alpha \int_0^t \left\{\sum_{j=1}^\infty F_j^2(\tau) m_j^{-2\beta{p}} (t-\tau)^{-2\alpha{p}} (t-\tau)^{2\alpha-2} \right\}^{1/2} d\tau.\label{Theo1aaa}
	\end{align}
	Hence, we obtain the following estimate
	\begin{align}
	\norm{\mathcal O_1 (t,.) F}  \le\,&  \widehat{c}_\alpha m_1^{- \beta{p}} \int_0^t \norm{F(\tau,.)} (t-\tau)^{\alpha{q}-1} d\tau \le M_1 t^{-\alpha {q} } \vertiii{F}_{\mathcal X _{2,\alpha{q}}}, \label{Theo1O1}
	\end{align}
	by noting (\ref{Fset}) and letting $M_1=\widehat{c}_\alpha m_1^{- \beta{p}}T^{\alpha{q}}$.  In addition, the norm $\norm{\mathcal O_{2}(t,x)\varphi}$ can be estimated as
	\begin{align}
	\norm{\mathcal O_{2}(t,.)\varphi}   =\,&    \left\{\sum_{j=1}^\infty  \varphi^2_j  \frac{E^2_{\alpha,1}(-m_j^\beta t^\alpha)}{E^2_{\alpha,1}(-m_j^\beta T^\alpha)} \right\}^{1/2}   \le \widehat{c}_\alpha c^{-1}_\alpha \left\{ \sum_{j=1}^\infty \varphi_j^2 \left[\frac{1+m_j^\beta T^\alpha}{1+m_j^\beta t^\alpha}\right]^2 \right\}^{1/2}. \nn
	\end{align}
	The ratio $\frac{1+m_j^\beta T^\alpha}{1+m_j^\beta t^\alpha}$ is clearly bounded by both $1+m_j^\beta T^\alpha$ and $\frac{T^\alpha}{t^\alpha}$. Moreover, the increasing property of the sequence $\{m_j\}_{j\ge 1}$ shows $1\le m_1^{-\beta} m_j^\beta$. Thus, we have  $1+m_j^\beta T^\alpha \le (m_1^{-\beta}+T^\alpha)m_j^\beta$. By  noting $p+q=1$, one can deduce that the ratio is bounded by the product of $T^{\alpha{q}} t^{-\alpha{q}}$ and $(1+m_j^\beta T^\alpha)^{{p}}$. Bring the above arguments together, and this leads to
	\begin{align}
	\norm{\mathcal O_{2}(t,.)\varphi} \le\,& \widehat{c}_\alpha c^{-1}_\alpha \left\{ \sum_{j=1}^\infty \varphi_j^2 \frac{T^{2\alpha{q}}}{t^{2\alpha{q}}} (1+m_j^\beta T^\alpha)^{2{p}} \right\}^{1/2} \le M_2  t^{-\alpha{q}} \norm{\varphi}_{{\bf V}_{\beta{p}}}, \label{Theo1O2}
	\end{align}
	where $$M_2=\widehat{c}_\alpha c^{-1}_\alpha  T^{\alpha {q}}   (m_1^{-\beta}+T^\alpha)^{{p}}.$$
	Now,  we proceed to estimate   $\norm{\mathcal O_{3}(t,x)\varphi}$ by using the same techniques as in (\ref{Theo1aaa}) and (\ref{Theo1O2}). As a consequence of
	\begin{align}
	\mathcal O_3(t,x)F = \,&  \mathcal O_2(t,x)\mathcal O_1(T,x)F    =  \sum_{j=1}^\infty   F_j(T)\star \widetilde{E}_{\alpha,\alpha}(-m_j^\beta T^\alpha)   \frac{E_{\alpha,1}(-m_j^\beta t^\alpha)}{E_{\alpha,1}(-m_j^\beta T^\alpha)}e_j(x) , \nn
	\end{align}
	we can obtain the following estimates
	\begin{align}
	\norm{\mathcal O_3(t,.)F} 
	\le\,& \int_0^T \norm{\sum_{j=1}^\infty F_j(\tau) E_{\alpha,\alpha}(-m_j^\beta (T-\tau)^\alpha) (T-\tau)^{\alpha-1} \frac{E_{\alpha,1}(-m_j^\beta t^\alpha)}{E_{\alpha,1}(-m_j^\beta T^\alpha)} e_j} d\tau \nn \\
	=\,& \int_0^T  \left\{\sum_{j=1}^\infty F_j^2(\tau) E^2_{\alpha,\alpha}(-m_j^\beta (T-\tau)^\alpha) (T-\tau)^{2\alpha-2} \frac{E^2_{\alpha,1}(-m_j^\beta t^\alpha)}{E^2_{\alpha,1}(-m_j^\beta T^\alpha)}   \right\}^{1/2} d\tau \nn\\
	\le\,& \widehat{c}_\alpha^2 c^{-1}_\alpha \int_0^T  \left\{\sum_{j=1}^\infty F_j^2(\tau)  m_j^{-2\beta{p}}  (T-\tau)^{2\alpha q-2} \frac{T^{2\alpha{q}}}{t^{2\alpha{q}}} (1+m_j^\beta T^\alpha)^{2{p}}   \right\}^{1/2} d\tau.  \nn
	\end{align}
	A simple computation shows that
	\begin{align}
	\norm{\mathcal O_3(t,.)F}\le\,&  M_3  t^{-\alpha{q}}\int_0^T \norm{F(\tau,.)} (T-\tau)^{\alpha{q}-1} d\tau \le M_3  t^{-\alpha{q}} \vertiii{F}_{\mathcal X _{2,\alpha{q}}}, \label{Theo1O3}
	\end{align}
	where we let $M_3=\widehat{c}_\alpha M_2$.  Finally, it follows from (\ref{Theo1O1}), (\ref{Theo1O2}),  (\ref{Theo1O3}), and the identity (\ref{mild})  that
	\begin{align}
	\norm{u(t,.)}\le\,& \norm{\mathcal O_1(t,.)F}+\norm{\mathcal O_2(t,.)\varphi}+\norm{\mathcal O_3(t,.)F} \hspace*{1.45cm} \nn\\
	\le\,& \left(\sum_{1\le n \le 3}M_n \right) \left(\norm{\varphi}_{{\bf V}_{\beta{p}}} + \vertiii{F}_{\mathcal X _{2,\alpha{q}}}\right)t^{-\alpha{q}}. \nn
	\end{align}
	The inequality (\ref{RLemma1}) is proved	by letting $\displaystyle C_0=\sum_{1\le n \le 3}M_n$.
\end{proof}

Based on  Lemma \ref{Lemma1}, we consider existence, uniqueness, and regularity of solutions in the following theorem which is divided into two parts. In the first part, we obtain the existence and uniqueness of a mild solution in  the space $L^{\frac{1}{\alpha q}-r}(0,T;L^2(D))$ for some suitable numbers $q$, $r$ and for the given  assumptions on $\varphi$ and $F$ as in Lemma \ref{Lemma1}. In the second part, we improve the smoothness of the mild solution by considering the spatial-fractional derivative $\mathcal L^{\beta(p-p')}$. It is very important to investigate the continuity of the mild solution. We first show that the mild solution is continuous from $(0,T]$ to $L^2(D)$. Moreover, we  establish the continuity on the closed interval $[0,T]$ which corresponds to lower spatial-smoothness, $V_{-\beta q'}$, for a relevant number $q'$.

\begin{theorem}\label{STheo1} \hspace*{20cm}\\
	a) Let $p,q,r$ be defined by (R1), (R2). If $\varphi\in {\bf V}_{\beta{p}}$  and $F\in \mathcal X _{2,\alpha{q}}(J\times D)$, then FVP (\ref{LINEAR}) has a unique solution $u$ in $L^{\frac{1}{\alpha q}-r}(0,T;L^2(D))$. Moreover, there exists a positive constant $C_1$ such that
	\begin{align}
	\norm{u}_{L^{\frac{1}{\alpha q}-r}(0,T;L^2(D))}  \le   C_1\norm{\varphi}_{{\bf V}_{\beta{p}}} + C_1\vertiii{F}_{\mathcal X _{2,\alpha{q}}}  . \label{RTheo1}
	\end{align}
	
	\vspace*{0.2cm}
	
	\noindent b) Let $p,q,s,r,p',q'$ be defined by (R1), (R3), (R4).  If $\varphi\in {\bf V}_{\beta{p}}$, and $F\in \mathcal X _{2,\alpha{q}- s }(J\times D)$, then FVP (\ref{LINEAR}) has a unique solution $u$ such that
	\begin{align}
	\,& u\in L^{\frac{1}{\alpha q'}-r}(0,T;{\bf V}_{\beta({p}- p' )})\cap C^{\alpha{q}}((0,T];L^2(D)) \cap C^{s}([0,T];{\bf V}_{-\beta {q}'}). \nn
	\end{align}
	Moreover, there exists a positive constant $C_2$ such that
	\begin{align}
	\,& \hspace*{-0.6cm} \norm{u}_{L^{\frac{1}{\alpha q'}-r}(0,T;{\bf V}_{\beta({p}-p')})} + \norm{u}_{C^{\alpha q}(0,T;L^2(D))} +  \vertiii{u}_{C^{s}([0,T];{\bf V}_{-\beta {q}'})} \label{RTheo2}  \\
	\,& \hspace*{4.6cm} \le  C_2\norm{\varphi}_{{\bf V}_{\beta{p}}} + C_2\vertiii{F}_{\mathcal X_{2,\alpha{q}-s}}  . \nn
	\end{align}
\end{theorem}

\begin{proof}  The proof of Part a) can be easily obtained  from Lemma \ref{Lemma1}. Indeed, the inequality (\ref{RLemma1}) leads  to
	\begin{align}
	\norm{u}_{L^{\frac{1}{\alpha q}-r}(0,T;L^2(D))}  \le   C_0\left(\norm{\varphi}_{{\bf V}_{\beta{p}}} + \vertiii{F}_{\mathcal X _{2,\alpha{q}}}\right) \left\{ \int_0^T t^{-\alpha{q}\left(\frac{1}{\alpha q}-r\right)}dt \right\}^{\frac{\alpha q}{1-\alpha q r}}. \nn
	\end{align}
	Since $-\alpha{q}\left(\frac{1}{\alpha q}-r\right)>-1$, the integral in the above inequality exists, i.e., $t^{-\alpha{q}}$ belongs to $L^{\frac{1}{\alpha q}-r}(0,T;\mathbb{R})$.  Hence, FVP (\ref{LINEAR}) has a solution $u$ in $L^{\frac{1}{\alpha q}-r}(0,T;L^2(D))$. The uniqueness of $u$ is obvious. Moreover, the inequality (\ref{RTheo1}) is derived  by  letting $C_1=C_0\norm{t^{-\alpha{q}}}_{L^{\frac{1}{\alpha q}-r}(0,T;\mathbb{R})}$. Now, we proceed to prove Part b) which will be presented in the following steps.
	
	\vspace*{0.2cm}
	
	\noindent {\bf Step 1:} We prove $u\in L^{\frac{1}{\alpha q}-r}(0,T;{\bf V}_{\beta({p}- p' )})$. Firstly, by the same argument as in the proof of (\ref{Theo1aaa}),    we derive the following chain of  inequalities
	\begin{align}
	\norm{\mathcal O_1 (t,.) F}_{{\bf V}_{\beta({p}- p' )}}
	\,&\le \int_0^t \norm{\mathcal L^{\beta({p}- p' )}\sum_{j=1}^\infty F_j(\tau) \widetilde{E}_{\alpha,\alpha}(-m_j^\beta (t-\tau)^\alpha)e_j} d\tau, \nn \\
	\,&\le \widehat{c}_\alpha \int_0^t \left\{\sum_{j=1}^\infty F_j^2(\tau)  m_j^{-2\beta{p}} (t-\tau)^{-2\alpha{p}} (t-\tau)^{2\alpha-2} m_j^{2\beta({p}- p' )} \right\}^{1/2} d\tau \hspace*{0.95cm} \nn\\
	\,&\le \widehat{c}_\alpha m_1^{-\beta p' } \int_0^t \norm{F(\tau,.)} (t-\tau)^{\alpha{q}-1} d\tau \le M_4 t^{-\alpha q' } \vertiii{F}_{\mathcal X_{2,\alpha{q}- s }},  \label{Theo2O1}
	\end{align}
	for  $M_4=\widehat{c}_\alpha m_1^{-\beta p' } T^{\alpha q' + s }$,   where the inequality $\vertiii{F}_{\mathcal X_{2,\alpha{q}}} \le T^{ s } \vertiii{F}_{\mathcal X_{2,\alpha{q}- s}}$ holds. Similarly, from  $\norm{\mathcal O_{2}(t,.)\varphi}_{{\bf V}_{\beta({p}- p' )}} =  \norm{\mathcal L^{\beta({p}- p' )} \mathcal O_{2}(t,.)\varphi}$, and  the same way as in the proof of (\ref{Theo1O2}), we have
	\begin{align}
	\,& \norm{\mathcal O_{2}(t,.)\varphi}_{{\bf V}_{\beta({p}- p' )}} \nn\\
	\le\,& \widehat{c}_\alpha c^{-1}_\alpha \left\{ \sum_{j=1}^\infty \varphi_j^2 \frac{T^{2\alpha q' }}{t^{2\alpha q' }} (1+m_j^\beta T^\alpha)^{2 p' } m_j^{2\beta({p}- p' )} \right\}^{1/2} \le M_5 t^{-\alpha q' }  \norm{\varphi}_{{\bf V}_{\beta{p}}}, \label{Theo2O2}
	\end{align}
	where we let $M_5=\widehat{c}_\alpha c^{-1}_\alpha T^{ \alpha q' } (m_1^{-\beta} +T^\alpha)^{ p' }$.  Now, we will estimate the norm $\norm{\mathcal O_3(t,.)F}_{{\bf V}_{\beta({p}- p' )}}$ which will use the same estimate for the fraction $\frac{E_{\alpha,1}(-m_j^\beta t^\alpha)}{E_{\alpha,1}(-m_j^\beta T^\alpha)}$ as in (\ref{Theo2O2}). Indeed, noting that $(1+m_j^\beta T^\alpha)^{ p' } \le (m_1^{-\beta}+ T^\alpha)^{ p' }  m_j^{\beta p' }$, by using (\ref{BasicMLInequality}), we see  that
	\begin{align}
	\,& \norm{\mathcal O_3(t,.)F}_{{\bf V}_{\beta({p}- p' )}} \nn\\
	\le\,& \int_0^T \norm{\mathcal L^{\beta({p}- p' )}\sum_{j=1}^\infty F_j(\tau) E_{\alpha,\alpha}(-m_j^\beta (T-\tau)^\alpha) (T-\tau)^{\alpha-1} \frac{E_{\alpha,1}(-m_j^\beta t^\alpha)}{E_{\alpha,1}(-m_j^\beta T^\alpha)} e_j} d\tau \nn \\
	=\,& \int_0^T  \left\{\sum_{j=1}^\infty F_j^2(\tau) E^2_{\alpha,\alpha}(-m_j^\beta (T-\tau)^\alpha) (T-\tau)^{2\alpha-2} \frac{E^2_{\alpha,1}(-m_j^\beta t^\alpha)}{E^2_{\alpha,1}(-m_j^\beta T^\alpha)}  m_j^{2\beta({p}- p' )} \right\}^{1/2} d\tau  \nn\\
	\le\,& \widehat{c}_\alpha M_5 \int_0^T  \left\{\sum_{j=1}^\infty  F_j^2(\tau)  m_j^{-2\beta{p}} (T-\tau)^{-2\alpha{p}}  (T-\tau)^{2\alpha-2}  t^{-2\alpha q' }  m_j^{2\beta p'} m_j^{2\beta({p}- p' )}  \right\}^{1/2} d\tau. \nn
	\end{align}
	Thus, by some simple computations, one can get
	\begin{align}
	\norm{\mathcal O_3(t,.)F}_{{\bf V}_{\beta({p}- p' )}} 
	\le\,& \widehat{c}_\alpha M_5 t^{-\alpha q' } \int_0^T  \norm{F(t,.)} (T-\tau)^{\alpha{q}-1} d\tau  \le M_6 t^{-\alpha q' } \vertiii{F}_{\mathcal X_{2,\alpha{q}-s}},\label{Theo2O3}
	\end{align}
	with $M_6=\widehat{c}_\alpha M_5 T^{s}$,	where  the inequality $\vertiii{F}_{\mathcal X_{2,\alpha{q}}} \le T^{ s } \vertiii{F}_{\mathcal X_{2,\alpha{q}- s}}$ has been used.  Bring (\ref{Theo2O1}), (\ref{Theo2O2}), (\ref{Theo2O3}) and (\ref{mild}) together, we have
	\begin{align}
	\norm{u(t,.)}_{{\bf V}_{\beta({p}- p' )}}  \le M_7\left(\norm{\varphi}_{V_{\beta{p}}} + \vertiii{F}_{\mathcal X_{2,\alpha{q}-s}}\right) t^{-\alpha q' },    \nn
	\end{align}
	where $ M_7=\sum_{4\le n\le 6} M_n$. Since the function $t \to    t^{-\alpha q' }$ is clearly contained in the space $L^{\frac{1}{\alpha q'}-r}(0,T;\mathbb{R})$, we can  take the $L^{\frac{1}{\alpha q'}-r}(0,T;\mathbb{R})$-norm on both sides of the above inequality, namely
	\begin{align}
	\norm{u}_{L^{\frac{1}{\alpha q'}-r}(0,T;{\bf V}_{\beta({p}-p')})}  \le   M_8\norm{\varphi}_{{\bf V}_{\beta{p}}} + M_8 \vertiii{F}_{\mathcal X_{2,\alpha{q}-s}}  , \label{RTheo2aaa}
	\end{align}
	where $M_8=M_7 \norm{t^{-\alpha{q}'}}_{L^{\frac{1}{\alpha q'}-r}(0,T;\mathbb{R})}$.
	
	\vspace*{0.2cm}
	\noindent {\bf Step 2:} We prove $u\in C^{\alpha{q}}((0,T];L^2(D))$. Let us consider   $0<t_1<t_2\le T$.  By (\ref{mild}), the difference $u(t_2,x)-u(t_1,x)$ can be calculated as
	\begin{align}
	\,& u(t_2,x)-u(t_1,x) \nn\\
	=\,&   \sum_{j=1}^\infty \left. F_j(t) \star \widetilde{E}_{\alpha,\alpha}(-m_j^\beta t^\alpha)\right|_{t=t_1}^{t=t_2} e_j(x) + \sum_{j=1}^\infty \varphi_j \left. \frac{E_{\alpha,1}(-m_j^\beta t^\alpha)}{E_{\alpha,1}(-m_j^\beta T^\alpha)} \right|_{t=t_1}^{t=t_2} e_j(x)  \nn\\
	-\,& \sum_{j=1}^\infty   F_j(T)\star \widetilde{E}_{\alpha,\alpha}(-m_j^\beta T^\alpha)   \left.\frac{E_{\alpha,1}(-m_j^\beta t^\alpha)}{E_{\alpha,1}(-m_j^\beta T^\alpha)} \right|_{t=t_1}^{t=t_2}e_j(x).\nn
	\end{align}
	By using the differentiation identities $$D_\omega   E_{\alpha ,1}(-m_j^\beta  \omega^\alpha )  =  -m_j^\beta \widetilde{E}_{\alpha ,\alpha  }(-m_j^\beta \omega^\alpha ) {~~\rm and~~}  D_\omega \big(\widetilde{E}_{\alpha ,\alpha  }(-m_j^\beta \omega^\alpha )\big) =  \omega^{\alpha -2}E_{\alpha ,\alpha -1}(-m_j^\beta \omega^\alpha ),$$ see, for example \cite{Diethelm,Podlubny,Samko}, we have
	\begin{align}
	\,& \left. F_j(t) \star \widetilde{E}_{\alpha,\alpha}(-m_j^\beta t^\alpha)\right|_{t=t_1}^{t=t_2} \nn\\
	=\,& \int_0^{t_1} F_j(\tau) \left.   \widetilde{E}_{\alpha,\alpha}(-m_j^\beta \omega^\alpha)\right|_{\omega=t_1-\tau}^{\omega=t_2-\tau} d\tau + \int_{t_1}^{t_2} F_j(\tau) \widetilde{E}_{\alpha,\alpha}(-m_j^\beta (t_2-\tau)^\alpha) d\tau \nn\\
	=\,& \int_0^{t_1} \int_{t_1-\tau}^{t_2-\tau}  F_j(\tau)\omega^{\alpha -2}E_{\alpha ,\alpha -1}(-m_j^\beta \omega^\alpha ) d\omega d\tau + \int_{t_1}^{t_2} F_j(\tau) \widetilde{E}_{\alpha,\alpha}(-m_j^\beta (t_2-\tau)^\alpha) d\tau,\nn
	\end{align}
	and
	$$\displaystyle \left.  {E_{\alpha,1}(-m_j^\beta t^\alpha)} \right|_{t=t_1}^{t=t_2}
	= -m_j^\beta \int_{t_1}^{t_2}   {\widetilde{E}_{\alpha ,\alpha  }(-m_j^\beta  \omega^\alpha )}  d\omega. $$
	Combining the above arguments gives
	\begin{align}
	\,& u(t_2,x)-u(t_1,x)\nn\\
	=\,&   \sum_{j=1}^\infty \int_0^{t_1} \int_{t_1-\tau}^{t_2-\tau}  F_j(\tau)\omega^{\alpha -2}E_{\alpha ,\alpha -1}(-m_j^\beta \omega^\alpha ) d\omega d\tau e_j(x) \nn\\
	+\,& \sum_{j=1}^\infty\int_{t_1}^{t_2}F_j(\tau) \widetilde{E}_{\alpha,\alpha}(-m_j^\beta (t_2-\tau)^\alpha) d\tau e_j(x) 
	-  \mathcal L^{\beta} \sum_{j=1}^\infty \varphi_j   \int_{t_1}^{t_2}  \frac{\widetilde{E}_{\alpha ,\alpha  }(-m_j^\beta  \omega^\alpha )}{E_{\alpha ,1}(-m_j^\beta  T^\alpha )} d\omega e_j(x)  \nn\\
	+ \,& \mathcal L^{\beta} \sum_{j=1}^\infty   \int_0^T  \int_{t_1}^{t_2}   F_j(\tau) \widetilde{E}_{\alpha,\alpha}(-m_j^\beta (T-\tau)^\alpha)  \frac{\widetilde{E}_{\alpha ,\alpha  }(-m_j^\beta  \omega^\alpha )}{E_{\alpha ,1}(-m_j^\beta  T^\alpha )} d\omega d\tau e_j(x)\nn\\
	:=\,& \mathcal I_1+\mathcal I_2 + \mathcal I_3  + \mathcal I_4.\label{ConFor}
	\end{align}
	Now, we will establish estimates for $\mathcal I_j$, $j=1,2,3,4$, and show that $\mathcal I_j$ tends to $0$ as $t_2 -t_1 \to 0$. Firstly, by the inequality (\ref{BasicMLInequality}), we see that the absolute value of $E_{\alpha ,\alpha -1}(-m_j^\beta \omega^\alpha )$ is bounded by $\widehat{c}_\alpha   m_j^{-\beta{p}} \omega^{-\alpha {p}}$. This implies
	\begin{align}
	\omega^{\alpha -2}\left| E_{\alpha ,\alpha -1}(-m_j^\beta \omega^\alpha ) \right| \le \widehat{c}_\alpha \omega^{\alpha q -2} m_j^{-\beta{p}}  .\nn
	\end{align}	
	Moreover, for $0<\tau<t_1$, we have  $$   \int_{t_1-\tau}^{t_2-\tau}  \omega^{\alpha {q} -2}  d\omega  =\frac{1}{1-\alpha {q}}\frac{(t_2-\tau)^{1-\alpha {q}}-(t_1-\tau)^{1-\alpha {q}} }{(t_1-\tau)^{1-\alpha {q}} (t_2-\tau)^{1-\alpha {q}} },$$
	where we note that the estimates
	$(t_2-\tau)^{1-\alpha {q}}-(t_1-\tau)^{1-\alpha {q}}  \le (t_2-t_1)^{1-\alpha {q}}$ and $(t_2-\tau)^{1-\alpha {q}} \ge (t_1-\tau)^{ s }(t_2-t_1)^{1-\alpha {q}- s }$ can be showed easily from $0<\alpha q < 1$, and $1-\alpha q -s >0$.
	Hence, we deduce
	\begin{align}
	\norm{\mathcal I_1}
	\le\,& \int_0^{t_1} \norm{ \sum_{j=1}^\infty  \int_{t_1-\tau}^{t_2-\tau}  F_j(\tau) \omega^{\alpha -2}E_{\alpha ,\alpha -1}(-m_j^\beta \omega^\alpha ) d\omega e_j}    d\tau \nn\\
	\le \,& \widehat{c}_\alpha m_1^{-\beta{p}} \int_0^{t_1}    \int_{t_1-\tau}^{t_2-\tau}  \omega^{\alpha q -2}    d\omega   \norm{F(\tau,.)} d\tau \nn\\
	\le \,& \frac{\widehat{c}_\alpha m_1^{-\beta{p}}}{1-\alpha {q}}  \int_0^{t_1} \norm{F(\tau,.)} (t_1-\tau)^{\alpha{q} -  s -1 }   d\tau (t_2-t_1)^{ s } . \nn
	\end{align}
	This leads to
	\begin{align}
	\,& \norm{\mathcal I_1}
	\le   M_9 \vertiii{F}_{\mathcal X_{2,\alpha{q}- s}} (t_2-t_1)^{ s }, \label{I1}
	\end{align}
	where we let $M_9=\frac{\widehat{c}_\alpha m_1^{-\beta{p}}}{1-\alpha {q}}$. Secondly, an estimate for the term $\mathcal I_2$ can be shown by using (\ref{BasicMLInequality}) as follows.
	\begin{align}
	\norm{\mathcal I_2}
	\le\,& \int_{t_1}^{t_2}\norm{\sum_{j=1}^\infty F_j(\tau) E_{\alpha,\alpha}(-m_j^\beta (t_2-\tau)^\alpha)  e_j} (t_2-\tau)^{\alpha-1} d\tau \nn\\
	\le\,& \widehat{c}_\alpha \int_{t_1}^{t_2}\norm{F(\tau,.)} (t_2-\tau)^{\alpha{q}- s -1}  (t_2-\tau)^{\alpha{p}+ s } d\tau \nn\\
	\le\,& M_{10}  \vertiii{F}_{\mathcal X_{2,\alpha{q}- s}} (t_2-t_1)^{ s }, \label{I2}
	\end{align}
	where we let $M_{10}=\widehat{c}_\alpha T^{\alpha p} $. Thirdly, we will estimate the term $\mathcal I_3$. We have
	$$\hspace*{-1.9cm} \displaystyle \norm{\mathcal I_3}  =\norm{\mathcal L^{\beta} \sum_{j=1}^\infty \varphi_j   \int_{t_1}^{t_2}  \frac{\widetilde{E}_{\alpha ,\alpha  }(-m_j^\beta  \omega^\alpha )}{E_{\alpha ,1}(-m_j^\beta  T^\alpha )} d\omega e_j}      . $$
	Here, the fraction can be estimated as follows
	\begin{align}
	\left|   \frac{\widetilde{E}_{\alpha ,\alpha  }(-m_j^\beta  \omega^\alpha )}{E_{\alpha ,1}(-m_j^\beta  T^\alpha )}  \right|  \le \widehat{c}_\alpha c_\alpha^{-1}   \left[\frac{1+m_j^\beta T^\alpha}{1+ m_j^\beta \omega^\alpha  }\right]^{{p}} \left[\frac{1+m_j^\beta T^\alpha}{1+ m_j^{2\beta} \omega^{2\alpha}  }\right]^{{q}}   \omega^{\alpha-1}  , \label{CI3aaa}
	\end{align}
	by  using (\ref{BasicMLInequality}) and  $\widetilde{E}_{\alpha ,\alpha  }(-m_j^\beta  \omega^\alpha )=E_{\alpha ,\alpha  }(-m_j^\beta  \omega^\alpha ) \omega^{\alpha-1}$. Moreover, we can see that
	$$\frac{1+m_j^\beta T^\alpha}{1+ m_j^{2\beta} \omega^{2\alpha}  } \le (m_1^{-\beta}+T^\alpha) m_j^\beta m_j^{-2\beta} \omega^{-2\alpha}\le  (m_1^{-\beta}+T^\alpha) m_j^{-\beta}\omega^{-2\alpha}.$$
	Taking these estimates together, we thus obtain the following chain of the inequalities
	\begin{align}
	\norm{\mathcal I_3}
	\le\,& \left\{ \sum_{j=1}^\infty \varphi_j^2 m_j^{2\beta} \left[ \int_{t_1}^{t_2}  \left| \frac{\widetilde{E}_{\alpha ,\alpha  }(-m_j^\beta  \omega^\alpha )}{E_{\alpha ,1}(-m_j^\beta  T^\alpha )} \right| d\omega \right]^2 \right\}^{1/2} \nn\\
	\le\,& M_{11} \left\{ \sum_{j=1}^\infty \varphi_j^2 m_j^{2\beta} \left[\int_{t_1}^{t_2} \omega^{-\alpha{p}} m_j^{-\beta{q}} \omega^{-2\alpha{q}} \omega^{\alpha-1} d\omega  \right]^2  \right\}^{1/2} \nn\\
	\le\,& M_{11} \left\{ \sum_{j=1}^\infty \varphi_j^2 m_j^{2\beta{p}} \left[\int_{t_1}^{t_2}     \omega^{-\alpha{q}-1}   d\omega  \right]^2  \right\}^{1/2}  ,  \nn
	\end{align}
	which implies that
	\begin{align}
	\norm{\mathcal I_3}
	\le M_{12} t_1^{-2\alpha q} (t_2-t_1)^{\alpha q}   \norm{\varphi}_{{\bf V}_{\beta{p}}},  \label{I3}
	\end{align}
	where $M_{11}= \widehat{c}_\alpha c_\alpha^{-1} T^{\alpha{p}} (m_1^{-\beta}+T^\alpha)^{{q}}$, and $M_{12}= M_{11} [\alpha {q}]^{-1} $. Fourthly, we proceed to estimate $\mathcal I_4$. According to (\ref{CI3aaa}), we have
	$ \left|   \frac{\widetilde{E}_{\alpha ,\alpha  }(-m_j^\beta  \omega^\alpha )}{E_{\alpha ,1}(-m_j^\beta  T^\alpha )}  \right| \le M_{11} m_j^{-\beta {q}} \omega^{-\alpha{q}-1}   $. Moreover, $\widetilde{E}_{\alpha,\alpha}(-m_j^\beta (T-\tau)^\alpha)   \le \widehat{c}_\alpha m_j^{-\beta {p}} (T-\tau)^{\alpha{q}-1}$  can be established by using the inequalities (\ref{BasicMLInequality}). Hence, we obtain
	\begin{align}
	\norm{\mathcal I_4}
	\le \,&\int_0^T \norm{\mathcal L^\beta\sum_{j=1}^\infty   \int_{t_1}^{t_2}   F_j(\tau) \widetilde{E}_{\alpha,\alpha}(-m_j^\beta (T-\tau)^\alpha)  \frac{\widetilde{E}_{\alpha ,\alpha  }(-m_j^\beta  \omega^\alpha )}{E_{\alpha ,1}(-m_j^\beta  T^\alpha )} d\omega   e_j}  d\tau  \nn\\
	\le\,& \int_0^T \left\{  \sum_{j=1}^\infty   m_j^{2\beta}F_j^2(\tau) \left[ \int_{t_1}^{t_2}   \left| \widetilde{E}_{\alpha,\alpha}(-m_j^\beta (T-\tau)^\alpha)  \frac{\widetilde{E}_{\alpha ,\alpha  }(-m_j^\beta  \omega^\alpha )}{E_{\alpha ,1}(-m_j^\beta  T^\alpha )} \right| d\omega \right]^2  \right\}^{1/2}  d\tau  \nn\\
	\le\,& \widehat{c}_\alpha M_{11} \int_0^T \left\{  \sum_{j=1}^\infty   m_j^{2\beta}F_j^2(\tau) \left[ \int_{t_1}^{t_2}   m_j^{-\beta {p}} (T-\tau)^{\alpha{q}-1}   m_j^{-\beta {q}} \omega^{-\alpha{q}-1}   d\omega \right]^2  \right\}^{1/2}  d\tau  \nn\\
	\le\,& \widehat{c}_\alpha M_{11} \frac{t_2^{\alpha {q}} - t_1^{\alpha {q}}}{t_1^{2\alpha {q}}} \int_0^T  \norm{F(\tau,.)}   (T-\tau)^{\alpha{q}-1}   d\tau , \nn
	\end{align}
	and we arrive at
	\begin{align}
	\norm{\mathcal I_4} \le \widehat{c}_\alpha M_{13} t_1^{-2\alpha q} (t_2-t_1)^{\alpha q}  \vertiii{F}_{\mathcal X_{2,\alpha{q}- s}}, \label{I4}
	\end{align}
	where $M_{13}=M_{11}T^s$. We deduce from (\ref{I1}), (\ref{I2}), \ref{I3}), (\ref{I4}) that  $\norm{\sum_{1\le j\le 4} \mathcal I_j}$ tends to $0$ as  $t_2-t_1$ tends $0$ for  $0<t_1<t_2\le T$. Thus, $u$ belongs to the set $C((0,T];L^2(D))$. On the other hand, by  assumption (R3), we have $0<\alpha q -s < \alpha q$ and
	$ \mathcal X _{2,\alpha{q}-s}(J\times D) \subset \mathcal X _{2,\alpha{q}}(J\times D).$ Therefore, the assumptions on $\varphi$ and $F$ in this theorem also fulfill Lemma \ref{Lemma1}. Hence, the
	inequality (\ref{RLemma1}) holds, i.e.,
	\begin{align}
	t^{\alpha{q}}\norm{u(t,.)}   \le M_4\left(\norm{\varphi}_{{\bf V}_{\beta{p}}} + \vertiii{F}_{2,\alpha{q}}\right), \quad t>0. \label{Casignom}
	\end{align}
	This implies $u$ belongs to $C^{\alpha{q}}((0,T];L^2(D))$. Moreover, by taking the supremum on both sides of (\ref{Casignom}) on $(0,T]$, we obtain
	\begin{align}
	\norm{u}_{C^{\alpha q}(0,T;L^2(D))}   \le\,& M_4\norm{\varphi}_{{\bf V}_{\beta{p}}} + T^s M_4\vertiii{F}_{\mathcal X_{2,\alpha{q}-s}}.  \label{RTheo2bbb}
	\end{align}
	
	\vspace*{0.2cm}
	\noindent {\bf Step 3:} We prove $u\in C^s([0,T];{\bf V}_{-\beta {q}'})$. In this step, we establish the continuity of the solution on the closed interval $[0,T]$. Now, we consider $0\le t_1 <t_2\le T$. If $t_1=0$, then $\mathcal I_1=0$. If $t_1>0$, then  combining (\ref{ggIP}) in the same way as in (\ref{I1}) gives  	
	\begin{align}
	\norm{\mathcal I_1}_{{\bf V}_{-\beta {q}'}}
	\le\,& \int_0^{t_1} \norm{ \sum_{j=1}^\infty  \int_{t_1-\tau}^{t_2-\tau}  F_j(\tau)\omega^{\alpha -2}E_{\alpha ,\alpha -1}(-m_j^\beta \omega^\alpha ) d\omega e_j}_{{\bf V}_{-\beta {q}'}}    d\tau \nn\\
	\le\,& \int_0^{t_1} \left\{ \sum_{j=1}^\infty m_j^{ -2\beta {q}' } \left( \int_{t_1-\tau}^{t_2-\tau}  F_j(\tau)\omega^{\alpha -2}E_{\alpha ,\alpha -1}(-m_j^\beta \omega^\alpha ) d\omega e_j, e_j \right)_{-\beta q',\beta q'}^2 \right\}^{1/2} d\tau \nn\\
	\le \,& \int_0^{t_1} \left\{ \sum_{j=1}^\infty m_j^{ -2\beta {q}' } F_j^2(\tau) \left| \int_{t_1-\tau}^{t_2-\tau}  \omega^{\alpha -2} \left| E_{\alpha ,\alpha -1}(-m_j^\beta \omega^\alpha )\right| d\omega \right|^2 \right\}^{1/2}    d\tau  , \nn
	\end{align}
	and so
	\begin{align}
	\norm{\mathcal I_1}_{{\bf V}_{-\beta {q}'}}
	\le \,&  \frac{\widehat{c}_\alpha m_1^{-\beta{q}'}}{1-\alpha {q}}  (t_2-t_1)^{ s } \vertiii{F}_{\mathcal X_{2,\alpha{q}- s}}.  \label{I1aaa}
	\end{align}
	On the other hand, the inequality (\ref{I2}) also holds for all $0\le t_1 < t_2 \le T$. Hence, the same way as in the proof (\ref{I2}) shows that
	\begin{align}
	\norm{\mathcal I_2}_{{\bf V}_{-\beta {q}'}}
&	\le \int_{t_1}^{t_2}\norm{\sum_{j=1}^\infty   F_j(\tau) E_{\alpha,\alpha}(-m_j^\beta (t_2-\tau)^\alpha)  e_j}_{{\bf V}_{-\beta {q}'}} (t_2-\tau)^{\alpha-1} d\tau \nn\\
	&\le m_1^{-\beta {q}'} \widehat{c}_\alpha (t_2-t_1)^{\alpha{p}+ s } \vertiii{F}_{\mathcal X_{2,\alpha{q}- s}}  \nn\\
	&\le m_1^{-\beta {q}'} \widehat{c}_\alpha T^{\alpha p} (t_2-t_1)^{s} \vertiii{F}_{\mathcal X_{2,\alpha{q}- s}} . \nn
	\end{align}
	Now, we will establish estimates for $\norm{I_3}_{{\bf V}_{-\beta {q}'}}$ and $\norm{I_4}_{{\bf V}_{-\beta {q}'}}$. Indeed, we have
	\begin{align}
	\left|   \frac{\widetilde{E}_{\alpha ,\alpha  }(-m_j^\beta  \omega^\alpha )}{E_{\alpha ,1}(-m_j^\beta  T^\alpha )}  \right|  \le\,& \widehat{c}_\alpha c_\alpha^{-1}   \left[\frac{1+m_j^\beta T^\alpha}{1+ m_j^\beta \omega^\alpha  }\right]^{{p}-{p}'} \left[\frac{1+m_j^\beta T^\alpha}{1+ m_j^{\beta} \omega^{\alpha}  }\right]^{{q}+{p}'}   \omega^{\alpha-1}  , \nn\\
	\le\,& \widehat{c}_\alpha c_\alpha^{-1}  (m_1^{-\beta}+T^\alpha)^{p-p'} m_j^{\beta (p-p')} T^{\alpha ({q}+{p}')}\omega^{-\alpha ({q}+{p}')} \omega^{\alpha-1} .\nn
	\end{align}
	Thus, we can derive 
	\[
	\left|   \frac{\widetilde{E}_{\alpha ,\alpha  }(-m_j^\beta  \omega^\alpha )}{E_{\alpha ,1}(-m_j^\beta  T^\alpha )}  \right| \le  M_{14} m_j^{\beta ({p}-{p}')} \omega^{\alpha ({p}-{p}') -1 },
	\]
	 where we let $$M_{14}= \widehat{c}_\alpha c_\alpha^{-1} (m_1^{-\beta}+T^\alpha)^{{p}-{p}'} T^{\alpha ({q}+{p}')}. $$
	 This leads to
	\begin{align}
	\norm{I_3}_{{\bf V}_{-\beta {q}'}}
	\le\,& \left\{ \sum_{j=1}^\infty \varphi_j^2 m_j^{2\beta} m_j^{-2\beta {q}'} \left[ \int_{t_1}^{t_2}  \left| \frac{\widetilde{E}_{\alpha ,\alpha  }(-m_j^\beta  \omega^\alpha )}{E_{\alpha ,1}(-m_j^\beta  T^\alpha )} \right| d\omega \right]^2 \right\}^{1/2} \nn\\
	\le\,& M_{14} \left\{ \sum_{j=1}^\infty \varphi_j^2 m_j^{2\beta} m_j^{-2\beta {q}'}  \left[\int_{t_1}^{t_2} m_j^{\beta ({p}-{p}')} \omega^{\alpha ({p}-{p}') -1 } d\omega  \right]^2  \right\}^{1/2} \nn\\
	\le\,& M_{14} \left\{ \sum_{j=1}^\infty \varphi_j^2 m_j^{2\beta{p}} \left[\int_{t_1}^{t_2}     \omega^{\alpha ({p}-{p}') -1 }  d\omega  \right]^2 \right\}^{1/2} .  \nn
	\end{align}
	Thus, by letting $ M_{15} = \frac{M_{14}}{\alpha({p}-{p}')}T^{\alpha(p-p')-s}$, we obtain the estimate
	\begin{align}
	\norm{I_3}_{{\bf V}_{-\beta {q}'}}  &\le  \frac{M_{14}}{\alpha({p}-{p}')} \norm{\varphi}_{{\bf V}_{\beta p}}  \left(t_2^{\alpha(p-p')}-t_1^{\alpha(p-p')} \right)\nn\\  &\le   M_{15} \norm{\varphi}_{{\bf V}_{\beta p}}  (t_2-t_1)^{s}  .  \label{I3COT}
	\end{align}
	where we have used that $$ t_2^{\alpha(p-p')}-t_1^{\alpha(p-p')} \le (t_2-t_1)^{\alpha(p-p')}\le T^{\alpha(p-p')-s} (t_2-t_1)^{s}$$ for $p'\le  p -\frac{s}{\alpha}$.  By  the same way as in (\ref{I3COT}), we have
	\begin{align}
	\norm{\mathcal I_4}_{{\bf V}_{-\beta {q}'}} 
	\le\,& \int_0^T \norm{\mathcal L^\beta\sum_{j=1}^\infty   \int_{t_1}^{t_2}   F_j(\tau) \widetilde{E}_{\alpha,\alpha}(-m_j^\beta (T-\tau)^\alpha)  \frac{\widetilde{E}_{\alpha ,\alpha  }(-m_j^\beta  \omega^\alpha )}{E_{\alpha ,1}(-m_j^\beta  T^\alpha )} d\omega   e_j}_{{\bf V}_{-\beta {q}'}}  d\tau  \nn\\
	\le\,& \int_0^T \left\{  \sum_{j=1}^\infty   m_j^{2\beta p'}  F_j^2(\tau) \left[ \int_{t_1}^{t_2}   \left| \widetilde{E}_{\alpha,\alpha}(-m_j^\beta (T-\tau)^\alpha)  \frac{\widetilde{E}_{\alpha ,\alpha  }(-m_j^\beta  \omega^\alpha )}{E_{\alpha ,1}(-m_j^\beta  T^\alpha )} \right| d\omega \right]^2  \right\}^{1/2}  d\tau  \nn\\
	\le\,& \widehat{c}_\alpha M_{14} \int_0^T \left\{  \sum_{j=1}^\infty    m_j^{2\beta p'} F_j^2(\tau) \left[ \int_{t_1}^{t_2}     (T-\tau)^{\alpha{q}-1}   m_j^{-\beta {p}' } \omega^{\alpha ({p}-{p}') -1 }   d\omega \right]^2  \right\}^{1/2}  d\tau  \nn\\
	\le\,& \widehat{c}_\alpha \frac{M_{14}}{\alpha({p}-{p}')}  \left(t_2^{\alpha(p-p')}-t_1^{\alpha(p-p')} \right) \int_0^T  \norm{F(\tau,.)}   (T-\tau)^{\alpha{q}-1}   d\tau ,  \nn
	\end{align}
	where $$|\widetilde{E}_{\alpha,\alpha}(-m_j^\beta (T-\tau)^\alpha)|\le m_j^{-\beta p}(T-\tau)^{\alpha q -1}.$$
		This implies that
	\begin{align}
	\norm{\mathcal I_4}_{{\bf V}_{-\beta {q}'}}
	\le  M_{16} (t_2 -t_1)^{s}  \vertiii{F}_{\mathcal X_{2,\alpha q -s}} , \label{I4COT}
	\end{align}		
	where $M_{16}=\widehat{c}_\alpha T^s M_{15}$. Combining the above arguments guarantees that $u$ belongs to $C^s([0,T];$ ${\bf V}_{-\beta {q}' })$. Moreover, there also exists a positive constant $M_{17}$ such that
	\begin{align}
	\vertiii{u}_{C^s([0,T];{\bf V}_{-\beta {q}' })}   \le M_{17} \norm{\varphi}_{{\bf V}_{\beta p}} + M_{17} \vertiii{F}_{\mathcal X_{2,\alpha q -s}}.\label{RTheo2ccc}
	\end{align}
	Finally,	the inequality (\ref{RTheo2}) is obtained by taking the inequality (\ref{RTheo2aaa}), (\ref{RTheo2bbb}) and (\ref{RTheo2ccc}) together.  The proof is complete.
\end{proof}

In the next theorem, we will investigate the time-space fractional derivative of the mild solution $u$. More specifically, we investigate ${}^{c}D_t^\alpha \mathcal L^{-\beta(q-\widehat{q})}u$, for a suitably chosen number $\widehat{q}\le q$. We also establish the continuity of ${}^{c}D_t^\alpha \mathcal L^{-\beta q}u$ on the interval $(0,T]$  without establishing it at $t=0$ since this requires a strong assumption of $F$, for example, $F$ must be continuous on whole interval $[0,T]$.

\begin{theorem}\label{STheo2} \hspace*{20cm}\\
	\noindent a) Let $p,q,s,p',q',\widehat{p},\widehat{q},r,\widehat{r}$ be defined by (R1), (R3), (R4), (R5). If $\varphi\in {\bf V}_{\beta({p}+\widehat{q})}$, and  $F\in L^{\frac{1}{\alpha q -s}+\widehat{r}}(0,T;L^2(D))$, then  FVP (\ref{LINEAR}) has a unique solution $u$ such that
	\begin{align}
	\,& u\in L^{\frac{1}{\alpha q'}-r}(0,T;{\bf V}_{\beta({p}- p' )})\cap C^{\alpha{q}}((0,T];L^2(D)) \cap C^{s}([0,T];{\bf V}_{-\beta {q}'}), \nn\\
	\,& \hspace*{2.8cm} {}^{c}D_t^\alpha u \in L^{\frac{1}{\alpha }-\widehat{r}}(0,T;{\bf V}_{-\beta(q-\widehat{q}) })  .\nn
	\end{align}
	Moreover, there exists a positive constant $C_3$ such that
	\begin{align}
	\norm{{}^{c}D_t^\alpha u }_{L^{\frac{1}{\alpha}-\widehat{r}}(0,T;{\bf V}_{-\beta(q-\widehat{q}) })}
	\le\,&  C_3 \norm{F}_{L^{\frac{1}{\alpha q -s } + \widehat{r}}(0,T;L^2(D))}    + C_3 \norm{\varphi}_{{\bf V}_{\beta(p+\widehat{q})}} . \label{RTheo2a}
	\end{align}	
	
	\vspace*{0.2cm}
	
	\noindent b) Let $p,q,s,p',q',\widehat{p},\widehat{q},r,\widehat{r}$ be defined by (R1), (R3), (R4), (R5). If $\varphi\in {\bf V}_{\beta({p}+\widehat{q})}$, and  $F\in L^{\frac{1}{\alpha q -s}+\widehat{r}}(0,T;L^2(D))\cap C^\alpha((0,T];V_{-\beta q}) $, then  FVP (\ref{LINEAR}) has a unique solution $u$ such that
	\begin{align}
	\,& u\in L^{\frac{1}{\alpha q'}-r}(0,T;{\bf V}_{\beta({p}- p' )})\cap C^{\alpha{q}}((0,T];L^2(D)) \cap C^{s}([0,T];{\bf V}_{-\beta {q}'}), \nn\\
	\,& \hspace*{1.3cm} {}^{c}D_t^\alpha u \in L^{\frac{1}{\alpha }-\widehat{r}}(0,T;{\bf V}_{-\beta(q-\widehat{q}) }) \cap C^\alpha((0,T];V_{-\beta q}) .\nn
	\end{align}
	Moreover, there exists a positive constant $C_4$ such that
	\begin{align}
	\,& \norm{{}^{c}D_t^\alpha u }_{L^{\frac{1}{\alpha}-\widehat{r}}(0,T;{\bf V}_{-\beta(q-\widehat{q}) })} + \norm{{}^{c}D_t^\alpha u}_{C^\alpha((0,T];{\bf V}_{-\beta q})} \nn\\
	\le\,& C_4 \norm{\varphi}_{{\bf V}_{\beta (p+\widehat{q}) }}  +     C_4 \norm{F}_{L^{\frac{1}{\alpha q -s } + \widehat{r}}(0,T;L^2(D))}  +  C_4\norm{F}_{C^{\alpha  }((0,T];{\bf V}_{-\beta {q}})} . \label{RTheo2b}
	\end{align}	
\end{theorem}

\begin{proof} a) By (R5), we have $0<\alpha q -s<1$, and $\frac{1}{\alpha q- s}+\widehat{r} > \frac{1}{\alpha q -s}>1$. Thus, one can deduce from (\ref{Holder})
	that
	\begin{align}
	L^{\frac{1}{\alpha q -s}+\widehat{r}}(0,T;L^2(D)) \subset \mathcal X _{2,\alpha q -s} (J\times D) ,\label{inclu1}
	\end{align}
	where we have used the inclusion (\ref{inclusion}). Moreover, the Sobolev embedding
	$
	{\bf V}_{\beta(p+\widehat{q})} \hookrightarrow {\bf V}_{\beta p}, $
	holds. Therefore, the assumptions of this theorem also fulfills Part b) of Theorem \ref{STheo1}. Hence, FVP (\ref{LINEAR}) has a unique solution
	\begin{align}
	\,& u\in L^{\frac{1}{\alpha q'}-r}(0,T;{\bf V}_{\beta({p}- p' )})\cap C^{\alpha{q}}((0,T];L^2(D)) \cap C^{s}([0,T];{\bf V}_{-\beta {q}'}). \nn
	\end{align}
	Now, we   prove   ${}^{c}D_t^\alpha u$ exists and belongs to  $L^{\frac{1}{\alpha  }-\widehat{r}}(0,T;{\bf V}_{-\beta(q-\widehat{q}) }) \cap C^{\alpha}((0,T];{\bf V}_{-\beta {q}}).$
	It follows from the identities
	\begin{align}
	{}^{c}D_t^\alpha  E_{\alpha ,1}(-m_j^\beta t^\alpha )  =\,& -m_j^\beta  E_{\alpha ,1 }(-m_j^\beta t^\alpha ), \nn\\
	{}^{c}D_t^\alpha \widetilde{E}_{\alpha ,\alpha  }(-m_j^\beta t^\alpha ) =\,& -m_j^\beta \widetilde{E}_{\alpha ,\alpha}(-m_j^\beta t^\alpha ),\nn
	\end{align}	
	see, for example \cite{Samko,Podlubny,Diethelm},
	and Equation (\ref{uj}) that
	\begin{align}
	{}^{c}D_t^\alpha u_j(t)
	=  \,& {}^{c}D_t^\alpha \left[ F_j(t)\star \widetilde{E}_{\alpha,\alpha}(-m_j^\beta t^\alpha) \right]  +  \Big[\varphi_j - F_j(T)\star \widetilde{E}_{\alpha,\alpha}(-m_j^\beta T^\alpha)  \Big] \frac{{}^{c}D_t^\alpha E_{\alpha,1}(-m_j^\beta t^\alpha)}{E_{\alpha,1}(-m_j^\beta T^\alpha)}    \nn\\
	=\,& F_j(t)  -m_j^\beta F_j(t)\star \widetilde{E}_{\alpha,\alpha}(-m_j^\beta t^\alpha)     -  \varphi_j  \frac{m_j^\beta E_{\alpha,1}(-m_j^\beta t^\alpha)}{E_{\alpha,1}(-m_j^\beta T^\alpha)} \nn\\
	+\,& F_j(T)\star \widetilde{E}_{\alpha,\alpha}(-m_j^\beta T^\alpha) \frac{m_j^\beta E_{\alpha,1}(-m_j^\beta t^\alpha)}{E_{\alpha,1}(-m_j^\beta T^\alpha)} := F_j(t)+\psi_j^{(1)}(t)+\psi_j^{(2)}(t)+\psi_j^{(3)}(t), \nn
	\end{align}
	for all $j\in \mathbb{N}$. Firstly, let us consider the sum $\sum_{n_1\le j\le n_2} \psi^{(1)}_j(t)e_j$, for  $n_1,n_2\in \mathbb{N}$, $1\le n_1<n_2$. By the definition of the dual space ${\bf V}_{-\beta(q-\widehat{q}) }$ of ${\bf V}_{\beta(q-\widehat{q}) }$, and the identity (\ref{ggIP}) of their dual inner product,  we have
	\begin{align}
	\,& \hspace*{0.53cm} \norm{\sum_{n_1\le j\le n_2} \psi^{(1)}_j(t)e_j}_{{\bf V}_{-\beta(q-\widehat{q}) }} \nn\\
	\le\,& \int_0^t \norm{   \sum_{n_1\le j\le n_2} m_j^\beta F_j(\tau) \widetilde{E}_{\alpha,\alpha}(-m_j^\beta (t-\tau)^\alpha) e_j }_{{\bf V}_{-\beta(q-\widehat{q}) }} d\tau \nn\\
	\le  \,& \int_0^t \left\{ \sum_{i=1}^\infty m_i^{2\beta (p+\widehat{q}) } \left( \sum_{n_1\le j\le n_2} F_j(\tau) \widetilde{E}_{\alpha,\alpha}(-m_j^\beta (t-\tau)^\alpha) e_j , e_i \right)_{-\beta(q-\widehat{q}), \beta(q-\widehat{q})}^2  \right\}^{1/2} d\tau \nn\\
	\le\,& \int_0^t  \left\{ \sum_{n_1\le j\le n_2} m_j^{2\beta (p+\widehat{q}) } F_j^2(\tau) \widetilde{E}^2_{\alpha,\alpha}(-m_j^\beta (t-\tau)^\alpha)  \right\}^{1/2} d\tau. \nn
	\end{align}
	Assumption (R5) shows that $0<p+\widehat{q}<1$. Hence, by using the inequalities (\ref{BasicMLInequality}),  we have $|\widetilde{E}_{\alpha,\alpha}(-m_j^\beta (t-\tau)^\alpha)| \le \widehat{c}_\alpha m_j^{-\beta (p+\widehat{q})} (t-\tau)^{-\alpha (p+\widehat{q}) } (t-\tau)^{\alpha-1}.$ This together with the above argument gives
	\begin{align}
	\hspace*{-0.35cm} \norm{\sum_{n_1\le j\le n_2} \psi^{(1)}_j(t)e_j}_{{\bf V}_{-\beta(q-\widehat{q}) }}  \le \,& M_{18} t^{-\alpha  }   \int_0^t (t-\tau)^{\alpha(q-\widehat{q})-1} \left\{ \sum_{n_1\le j\le n_2}  F_j^2(\tau)   \right\}^{1/2} d\tau, \label{psi1}
	\end{align}
	where $M_{18}=\widehat{c}_\alpha T^{\alpha  }$.
	Secondy, we proceed to establish an estimate for the sum $  \sum_{n_1\le j\le n_2} \psi^{(2)}_j(t)e_j$. Using the inequality (\ref{BasicMLInequality}), the absolute value of $\frac{ E_{\alpha,1}(-m_j^\beta t^\alpha)}{E_{\alpha,1}(-m_j^\beta T^\alpha)}$ is bounded by $\widehat{c}_\alpha c_\alpha^{-1} T^{\alpha} t^{-\alpha}$. Therefore, we derive
	\begin{align}
	\hspace*{1.25cm} \norm{\sum_{n_1\le j\le n_2} \psi^{(2)}_j(t)e_j}_{{\bf V}_{-\beta(q-\widehat{q}) }} =\,& \left\{ \sum_{n_1\le j\le n_2}  m_j^{2\beta (p+\widehat{q}) } \varphi_j^2 \frac{ E^2_{\alpha,1}(-m_j^\beta t^\alpha)}{E^2_{\alpha,1}(-m_j^\beta T^\alpha)} \right\}^{1/2} \nn \\
	\le\,&  \widehat{c}_\alpha c_\alpha^{-1} T^{\alpha}  \left\{ \sum_{n_1\le j\le n_2}  m_j^{2\beta (p+\widehat{q}) } \varphi_j^2    t^{-2\alpha } \right\}^{1/2}, \nn
	\end{align}
	which shows that
	\begin{align}
	\hspace*{-0.55cm} \norm{\sum_{n_1\le j\le n_2} \psi^{(2)}_j(t)e_j}_{{\bf V}_{-\beta(q-\widehat{q}) }}  \le \,& M_{19} t^{-\alpha}  \left\{ \sum_{n_1\le j\le n_2}  m_j^{2\beta (p+\widehat{q}) } \varphi_j^2    \right\}^{1/2} , \label{psi2}
	\end{align}
	where $M_{19}:=\widehat{c}_\alpha c_\alpha^{-1} T^{\alpha}$.
	Thirdly, we proceed to establish an estimate for the sum $\sum_{n_1\le j\le n_2} \psi^{(3)}_j(t)e_j$. By a similar argument as in (\ref{psi2}), we have
	\begin{align}
	\,& \hspace*{0.63cm} \norm{\sum_{n_1\le j\le n_2} \psi^{(3)}_j(t)e_j}_{{\bf V}_{-\beta(q-\widehat{q}) }} \nn\\
	\le \,& \int_0^T \norm{\sum_{n_1\le j\le n_2} F_j(\tau) \widetilde{E}_{\alpha,\alpha}(-m_j^\beta (T-\tau)^\alpha) \frac{m_j^\beta E_{\alpha,1}(-m_j^\beta t^\alpha)}{E_{\alpha,1}(-m_j^\beta T^\alpha)} e_j}_{{\bf V}_{-\beta(q-\widehat{q}) }} d\tau  \nn\\
	\le \,& \int_0^T \left\{\sum_{n_1\le j\le n_2} m_j^{2\beta (p+\widehat{q}) } F_j^2(\tau)  \widetilde{E}^2_{\alpha,\alpha}(-m_j^\beta (T-\tau)^\alpha) \frac{ E^2_{\alpha,1}(-m_j^\beta t^\alpha)}{E^2_{\alpha,1}(-m_j^\beta T^\alpha)}  \right\}^{1/2} d\tau \nn\\
	\le \,& \widehat{c}_\alpha M_{19} t^{-\alpha }   \int_0^T \left\{\sum_{n_1\le j\le n_2} m_j^{2\beta (p+\widehat{q}) } F_j^2(\tau)  m_j^{-2\beta (p+\widehat{q})} (T-\tau)^{2\alpha(q-\widehat{q})-2}    \right\}^{1/2} d\tau .  \nn
	\end{align}
	Therefore, we obtain the following estimate
	\begin{align}
	\hspace*{-0.5cm} \norm{\sum_{n_1\le j\le n_2} \psi^{(3)}_j(t)e_j}_{{\bf V}_{-\beta(q-\widehat{q}) }}
	\le \,&   \widehat{c}_\alpha M_{19} t^{-\alpha  }  \int_0^T   (T-\tau)^{\alpha(q-\widehat{q}) - 1} \left\{ \sum_{n_1\le j\le n_2}  F_j^2(\tau) \right\}^{1/2}  d\tau.  \label{psi3}
	\end{align}
	For almost every $\tau$ in the interval $(0,T)$, by (\ref{inclu1}), we have that $F(\tau,.)$ belongs to $L^2(D)$. This implies  $\sum_{1\le j\le n}  F_j(\tau)e_j$ is a Cauchy sequence in $L^2(D)$. This together with the embedding $$L^2(D) \hookrightarrow {\bf V}_{-\beta(q-\widehat{q}) }$$
	implies that $\sum_{1\le j\le n }  F_j(\tau)e_j$ is also a Cauchy sequence in ${\bf V}_{-\beta(q-\widehat{q}) }$. On the other hand, it follows from $\varphi\in {\bf V}_{\beta (p+\widehat{q})}$ that
	$$
	\lim\limits_{n_1,n_2\to \infty}   \sum_{n_1\le j\le n_2}   \varphi_j^2  m_j^{2\beta (p+\widehat{q})} = 0
	.$$
	By (R5), $0\le \widehat{q} \le \frac{s}{\alpha}$, and we obtain the inclusion  $ \mathcal X _{2,\alpha q - s}(J\times D) \subset \mathcal X _{2,\alpha({q}- \widehat{q})}(J\times D)$. We deduce that $F\in \mathcal X _{2,\alpha({q}- \widehat{q})}(J\times D)$, and
	\begin{align}
	\lim\limits_{n_1,n_2\to \infty} \int_0^T   (T-\tau)^{\alpha(q-\widehat{q}) - 1} \left\{ \sum_{n_1\le j\le n_2}  F_j^2(\tau) \right\}^{1/2}  d\tau =0,
	\end{align}	
	by the dominated convergence theorem.
	Combining these with the estimates (\ref{psi1}), (\ref{psi2}) and (\ref{psi3}), we have that
	$$\lim\limits_{n_1,n_2\to \infty} \norm{ \sum_{n_1\le j\le n_2} {}^{c}D_t^\alpha u_j(t) e_j}_{{\bf V}_{-\beta(q-\widehat{q}) }} =0 .$$
	Hence $\sum_{j=1}^n {}^{c}D_t^\alpha u_j(t) e_j $ is a Cauchy sequence and a convergent sequence in ${\bf V}_{-\beta(q-\widehat{q}) }$. We then conclude that ${}^{c}D_t^\alpha u(t,.)=\sum_{j=1}^\infty {}^{c}D_t^\alpha u_j(t) e_j $ finitely exists in the space  ${\bf V}_{-\beta(q-\widehat{q}) }$. Moreover, by taking the inequalities (\ref{psi1}), (\ref{psi2}) and (\ref{psi3}),  there exists a constant $M_{20}>0$ such that  	
	\begin{align}
	\norm{{}^{c}D_t^\alpha u(t,.)}_{{\bf V}_{-\beta(q-\widehat{q}) }}
	\le\,& M_{20} \norm{F(t,.)}_{-\beta(q-\widehat{q}) }    + M_{20} \left(  \norm{\varphi}_{{\bf V}_{\beta(p+\widehat{q})}}    +   \vertiii{F}_{\mathcal X_{2,\alpha({q}- \widehat{q})}} \right)t^{-\alpha  } \nn \\
	\le\,& M_{20} \norm{F(t,.)}   + M_{20} \left(  \norm{\varphi}_{{\bf V}_{\beta(p+\widehat{q})}}    +   \vertiii{F}_{\mathcal X_{2,\alpha({q}- \widehat{q})}} \right)t^{-\alpha  }. \label{over}
	\end{align}	
	Now, it follows from $0<\widehat{r}\le \frac{1-\alpha}{\alpha}$ and $0<\alpha q -s <\alpha$ that $1\le \frac{1}{\alpha}-\widehat{r} < \frac{1}{\alpha q -s } + \widehat{r}$. This implies the following Sobolev embedding
	$$L^{\frac{1}{\alpha q -s}+\widehat{r}}(0,T;{\bf V}_{-\beta(q-\widehat{q}) }) \hookrightarrow L^{\frac{1}{\alpha}-\widehat{r}}(0,T;{\bf V}_{-\beta(q-\widehat{q}) }).$$ Moreover, by the assumption (R5), $\widehat{q}<\frac{s}{\alpha}$, we have $ \frac{1}{\alpha q -s}+\widehat{r}>\frac{1}{\alpha (q-\widehat{q})}$. This implies that there exists a constant $C_*>0$ such that
	\begin{align}
	\vertiii{F}_{\mathcal X_{2,\alpha({q}- \widehat{q})}} \le C_* \norm{F}_{L^{\frac{1}{\alpha q -s } + \widehat{r}}(0,T;L^2(D))}. \label{mid}
	\end{align}
	Hence, we deduce from (\ref{over}) that there exists a constant $M_{21}>0$ satisfying
	\begin{align}
	\norm{{}^{c}D_t^\alpha u }_{L^{\frac{1}{\alpha}-\widehat{r}}(0,T;{\bf V}_{-\beta(q-\widehat{q}) })}
	\le\,&  M_{21} \norm{F}_{L^{\frac{1}{\alpha q -s } + \widehat{r}}(0,T;L^2(D))}    + M_{21}   \norm{\varphi}_{{\bf V}_{\beta(p+\widehat{q})}}      , \label{down}
	\end{align}	
	where we note that $\norm{t^{-\alpha}}_{L^{\frac{1}{\alpha}-\widehat{r}}(0,T;\mathbb{R})}<\infty$. The inequality (\ref{RTheo2a}) is proved by letting $C_3=M_{21}$.
	
	\vspace*{0.2cm}
	
	\noindent b) It is clear that the assumptions of this part also satisfy Part a. Therefore, by Part a, it is necessary to prove  $u\in C^\alpha((0,T];{\bf V}_{-\beta {q}})$, i.e.,
	\begin{align}
	\lim\limits_{t_2-t_1\to 0}\norm{{}^{c}D_t^\alpha u(t_2,.)-{}^{c}D_t^\alpha u(t_1,.)}_{{\bf V}_{-\beta {q}}}=0, \label{34b}
	\end{align}
	where we note $0<t_1<t_2\le T$.  After some simple computations we find that
	\begin{align}
	{}^{c}D_t^\alpha u(t_2,x)-{}^{c}D_t^\alpha u(t_1,x) = F(t_2,x)-F(t_1,x) + \sum_{1\le n \le 4} \mathcal{J}_n,\label{J}
	\end{align}
	where $\mathcal J_n =\mathcal L^\beta \mathcal I_n$, and $\mathcal I_n$ is defined by (\ref{ConFor}). Since $F$ is in $C^\alpha((0,T];{\bf V}_{-\beta {q}})$, we  have just to prove $\norm{\mathcal J_n}_{{\bf V}_{-\beta {q}}}$ approaches $0$ as $t_2-t_1$ approaches $0$. Let us first consider $\norm{\mathcal J_1}_{{\bf V}_{-\beta {q}}}$. The inequalities (\ref{BasicMLInequality}) yields that
	\begin{align}
	\norm{\mathcal J_1}_{{\bf V}_{-\beta {q}}}
	\le\,& \int_0^{t_1} \norm{ \mathcal L^\beta \sum_{j=1}^\infty  \int_{t_1-\tau}^{t_2-\tau}  F_j(\tau)\omega^{\alpha -2}E_{\alpha ,\alpha -1}(-m_j^\beta \omega^\alpha ) d\omega e_j}_{{\bf V}_{-\beta {q}}}    d\tau \nn\\
	\le \,& \int_0^{t_1} \left\{ \sum_{j=1}^\infty m_j^{2\beta} m_j^{ -2\beta {q}  } F_j^2(\tau) \left| \int_{t_1-\tau}^{t_2-\tau}  \omega^{\alpha -2} \left| E_{\alpha ,\alpha -1}(-m_j^\beta \omega^\alpha )\right| d\omega \right|^2 \right\}^{1/2}    d\tau \hspace*{-1cm} \nn\\
	\le \,& \widehat{c}_\alpha \int_0^{t_1} \left\{ \sum_{j=1}^\infty m_j^{2\beta} m_j^{ -2\beta {q}  } F_j^2(\tau) \left| \int_{t_1-\tau}^{t_2-\tau}  \omega^{\alpha -2}   m_j^{-\beta p} \omega^{-\alpha p}   d\omega \right|^2 \right\}^{1/2}    d\tau. \nn
	\end{align}
	We recall that, by (\ref{inclu1}), $F$ belongs to $\mathcal X _{2,\alpha q -s} (J\times D)$. 	Thus, we can deduce from the above inequality  that
	\begin{align}
	\norm{\mathcal J_1}_{{\bf V}_{-\beta {q}}}
	\le \,& \widehat{c}_\alpha \int_0^{t_1} \norm{F(\tau,.)} \left| \int_{t_1-\tau}^{t_2-\tau}  \omega^{\alpha q -2}        d\omega \right|     d\tau \le \frac{\widehat{c}_\alpha  }{1-\alpha {q}}  \vertiii{F}_{\mathcal X_{2,\alpha q -s} }  (t_2-t_1)^{ s }. \hspace*{0.65cm} \label{J1}
	\end{align}
	where we have used the same argument as in the estimate (\ref{I1}). Let us secondly consider $\norm{\mathcal J_2}_{{\bf V}_{-\beta {q}}}$. We have
	\begin{align}
	\norm{\mathcal J_2}_{{\bf V}_{-\beta {q}}}
	\le\,& \int_{t_1}^{t_2}\norm{\mathcal L^\beta \sum_{j=1}^\infty F_j(\tau) E_{\alpha,\alpha}(-m_j^\beta (t_2-\tau)^\alpha)  e_j}_{V_{-\beta {q}}} (t_2-\tau)^{\alpha-1} d\tau \nn\\
	\le\,& \widehat{c}_\alpha \int_{t_1}^{t_2} \left\{ \sum_{j=1}^\infty m_j^{2\beta} m_j^{ -2\beta {q}  } F_j^2(\tau) m_j^{-2\beta p} (t_2-\tau)^{-2\alpha p} \right\}^{1/2}  (t_2-\tau)^{\alpha-1} d\tau \nn \\
	\le\,& \widehat{c}_\alpha \int_{t_1}^{t_2} \norm{F(\tau,.)} (t_2-\tau)^{\alpha q -s -1} (t_2-\tau)^{s}d\tau \le \widehat{c}_\alpha \vertiii{F}_{\mathcal X_{2,\alpha q -s} }(t_2-t_1)^{s} , \label{J2}
	\end{align}
	where (\ref{BasicMLInequality}) has been used. Thirdly, we  consider the norm $\norm{\mathcal J_3}_{V_{-\beta {q}}}$. It is clear that
	 $$\mathcal J_3 = \mathcal L^\beta \mathcal I_3 = \mathcal L^{2\beta} \sum_{j=1}^\infty \varphi_j   \int_{t_1}^{t_2}  \frac{\widetilde{E}_{\alpha ,\alpha  }(-m_j^\beta  \omega^\alpha )}{E_{\alpha ,1}(-m_j^\beta  T^\alpha )} d\omega e_j.$$ 
	Hence, we deduce
	\begin{align}
	\norm{\mathcal J_3}_{{\bf V}_{-\beta {q}}}   = \left\{ \sum_{j=1}^\infty m_j^{4\beta} m_j^{ -2\beta {q}  } \varphi_j^2    \left| \int_{t_1}^{t_2}  \frac{\widetilde{E}_{\alpha ,\alpha  }(-m_j^\beta  \omega^\alpha )}{E_{\alpha ,1}(-m_j^\beta  T^\alpha )} d\omega \right|^2  \right\}^{1/2} .\nn
	\end{align}
	By applying (\ref{BasicMLInequality}), the absolute value of $\frac{E_{\alpha ,\alpha  }(-m_j^\beta  \omega^\alpha )}{E_{\alpha ,1}(-m_j^\beta  T^\alpha )}$ is bounded by $\widehat{c}_\alpha c_\alpha^{-1} \frac{1+m_j^\beta  T^\alpha }{1+\left(m_j^{\beta}\omega^{\alpha}\right)^2}$. This is associated with $1+m_j^\beta  T^\alpha \le (m_1^{-\beta}+T^\alpha) m_j^\beta$ that $\frac{1+m_j^\beta  T^\alpha }{1+\left(m_j^{\beta}\omega^{\alpha}\right)^2} \le (m_1^{-\beta}+T^\alpha) m_j^{-\beta} \omega^{-2\alpha}$. Thus,  we obtain
	\begin{align}
	\norm{\mathcal J_3}_{{\bf V}_{-\beta {q}}}
	\le\,& \widehat{c}_\alpha c_\alpha^{-1} (m_1^{-\beta}+T^\alpha) \left\{ \sum_{j=1}^\infty m_j^{4\beta} m_j^{ -2\beta {q}  } \varphi_j^2    \left| \int_{t_1}^{t_2}  m_j^{-\beta}\omega^{-2\alpha}\omega^{\alpha -1} d\omega \right|^2  \right\}^{1/2} \nn \\
	\le\,& \frac{M_{23}}{\alpha} t_1^{-2\alpha}  \norm{\varphi}_{{\bf V}_{\beta p}} \left( t_2^\alpha - t_1^\alpha \right), \label{J3}
	\end{align}
	where $M_{23}= \widehat{c}_\alpha c_\alpha^{-1} (m_1^{-\beta}+T^\alpha)  $. Finally, we can look at  $\norm{\mathcal J_4}_{{\bf V}_{-\beta {q}}}$   as follows:
	\begin{align}
	\norm{\mathcal J_4}_{{\bf V}_{-\beta {q}}}
	\le  \,& \int_0^T \norm{ \mathcal L^{2\beta}\sum_{j=1}^\infty   \int_{t_1}^{t_2}   F_j(\tau) \widetilde{E}_{\alpha,\alpha}(-m_j^\beta (T-\tau)^\alpha)  \frac{\widetilde{E}_{\alpha ,\alpha  }(-m_j^\beta  \omega^\alpha )}{E_{\alpha ,1}(-m_j^\beta  T^\alpha )} d\omega   e_j }_{{\bf V}_{-\beta {q}}}  d\tau  \nn\\
	\le \,&  \int_0^T \left\{ \sum_{j=1}^\infty   m_j^{4\beta} m_j^{ -2\beta {q}  } F_j^2(\tau) \left| \int_{t_1}^{t_2}    \widetilde{E}_{\alpha,\alpha}(-m_j^\beta (T-\tau)^\alpha)  \frac{\widetilde{E}_{\alpha ,\alpha  }(-m_j^\beta  \omega^\alpha )}{E_{\alpha ,1}(-m_j^\beta  T^\alpha )} d\omega \right|^2  \right\}^{1/2}  d\tau   \nn \\
	\le \,& M_{24} \int_0^T \left\{ \sum_{j=1}^\infty   m_j^{4\beta} m_j^{ -2\beta {q}  } F_j^2(\tau) \left| \int_{t_1}^{t_2}  m_j^{- \beta p} (T-\tau)^{ \alpha q -1 } m_j^{-\beta} \omega^{-\alpha-1}    d\omega \right|^2  \right\}^{1/2}  d\tau   \nn
	\end{align}
	where the fraction $\frac{\widetilde{E}_{\alpha ,\alpha  }(-m_j^\beta  \omega^\alpha )}{E_{\alpha ,1}(-m_j^\beta  T^\alpha )}$ can be estimated in the same way as in the proof of (\ref{J3}), and $M_{24}=\widehat{c}_\alpha M_{23}$. This leads to
	\begin{align}
	\norm{\mathcal J_4}_{{\bf V}_{-\beta {q}}}
	\le \,&  \frac{M_{24}}{\alpha} t_1^{-2\alpha} (t_2^\alpha -t_1^\alpha) \int_0^T \norm{F(\tau,.)} (T-\tau)^{ \alpha q -1 }  d\tau,  \nn
	\end{align}	
	which shows that
	\begin{align}
	\norm{\mathcal J_4}_{{\bf V}_{-\beta {q}}} \le\,& \frac{M_{24}}{\alpha} t_1^{-2\alpha} \vertiii{F}_{\mathcal X_{2,\alpha q} } (t_2^\alpha -t_1^\alpha) . \label{J4}
	\end{align}
	This implies (\ref{34b}) by taking (\ref{J}), (\ref{J1}), (\ref{J2}), (\ref{J3}) and (\ref{J4}) together. Thus, $^{c}D_t^\alpha u$ is contained in  $ C((0,T];{\bf V}_{-\beta {q}})$. \\
	
	On the other hand, it is easy to see that the estimates (\ref{psi1}), (\ref{psi2}), (\ref{psi3}) also hold for $\widehat{q}=0$. Hence, we deduce from (\ref{over}) and (\ref{mid}) that
	\begin{align}
	\,& t^\alpha \norm{{}^{c}D_t^\alpha u(t,.)}_{{\bf V}_{-\beta q}} \nn\\
	\le\,& M_{20}t^\alpha \norm{F(t,.)}_{{\bf V}_{-\beta q}}   + M_{20} \left(  \norm{\varphi}_{{\bf V}_{\beta p}}    +    C_*  \norm{F}_{L^{\frac{1}{\alpha q -s } + \widehat{r}}(0,T;L^2(D))}  \right) .
	\end{align}
	Now ${}^{c}D_t^\alpha u \in C^\alpha((0,T];{\bf V}_{-\beta {q}})$. In addition, there exists a positive constant $C'>0$ such that
	\begin{align}
	\,& \norm{{}^{c}D_t^\alpha u}_{C^\alpha((0,T];{\bf V}_{-\beta q})} \nn\\
	\le\,&  M_{20}\norm{F}_{C^{\alpha  }((0,T];{\bf V}_{-\beta {q}})} +   C'M_{20} \norm{\varphi}_{{\bf V}_{\beta (p+\widehat{q}) }} +  C_*M_{20} \norm{F}_{L^{\frac{1}{\alpha q -s } + \widehat{r}}(0,T;L^2(D))}. \nn
	\end{align}
	We can complete the proof by taking (\ref{RTheo2a}) and the above inequality together.
\end{proof}

\section{FVP with a nonlinear source}

\noindent In this section, we study the existence, uniqueness, and regularity of mild solutions of FVP (\ref{mainpro1})-(\ref{mainpro3}) corresponding to the nonlinear source function $F(t,x,u(t,x))$. It is suitable  considering assumptions that $u(t,.)$ and $F(t,.,u(t,.))$ belong to the same spatial space $H$. In view of most considerations of PDEs, we let $H=L^2(D)$.

\vspace*{0.2cm}

We introduce the following assumptions on the numbers $p,q,p',q',\widehat{p},\widehat{q},r,\widehat{r}$. \vspace*{0.2cm}
\begin{itemize}
	\item (R1b) $0<q<p<1$ such that ${p}+{q}=1$; \vspace*{0.1cm}
	\item (R4b) $
	\displaystyle 0 < p' <p , \hspace*{0.9cm} q' =1- p' , \quad 0<r\le \frac{1-\alpha q'}{\alpha q'}
	$; \vspace*{0.1cm}
	\item (R4c) $
	\displaystyle 0 < p' \le p-q, \quad q' =1- p' , \quad 0<r\le \frac{1-\alpha q'}{\alpha q'}
	$; \vspace*{0.1cm}
	\item (R5b) $
	\displaystyle 0 \le \widehat{q} < q, \hspace*{1.12cm} \widehat{p} =1- \widehat{q} , \hspace*{0.45cm} 0< \widehat{r} \le \frac{1-\alpha  }{\alpha}
	$.  	\vspace*{0.1cm}
\end{itemize}

In our work, we will assume on $F(t,.,u(t,.))$ the following assumptions \vspace*{0.2cm}
\begin{itemize}
	\item (A1) $F(t,.,\mathbf 0)=\mathbf 0$, and there exists a constant $K>0$ such that, for all  $v_1,v_2 \in L^2(D)$ and $t\in J$,
	\begin{align}
	\norm{F(t,.,v_1)-F(t,.,v_2)} \le K \norm{v_1-v_2}. \nn
	\end{align}
	\item (A2) $F(t,.,\mathbf 0)=\mathbf 0$, and there exists a constant $K_*>0$ such that, for all $v_1,v_2 \in L^2(D)$ and $t_1,t_2\in J$,
	\begin{align}
	\norm{F(t_1,.,v_1)-F(t_2,.,v_2)} \le K_*\left(|t_1-t_2|+ \norm{v_1-v_2} \right). \nn
	\end{align}
\end{itemize}

Note that the assumption (A1), (A2) imply that, for
$v\in L^2(D)$,
\begin{align}
\norm{F(t,.,v)} \le K \norm{v}. \label{FvKv}
\end{align}

\vspace*{0.2cm}

We try to develop the ideas of the linear FVP (\ref{LINEAR}) to deal with the nonlinear FVP (\ref{mainpro1})-(\ref{mainpro3}). In Section 3, for the linear function $F(t,x)$ we assume that
\begin{align}
F\in \mathcal X _{2,\alpha{q}}(J\times D),   \textrm{ or }  F\in \mathcal X _{2,\alpha{q}-s}(J\times D),  \textrm{ or } F \in L^{\frac{1}{\alpha q -s}+\widehat{r}}(0,T;L^2(D)), \label{LinearFA}
\end{align}
where $p,q,s,\widehat{r}$ are defined by (R1), (R3), (R5). However, we cannot suppose that the nonlinear source function $F(t,x,u(t,x))$ satisfies the same assumptions as in (\ref{LinearFA}), and then find the solution $u$.  A natural  idea might be to combine the idea of Lemma \ref{Lemma1} with the inequality (\ref{FvKv}), i.e., we predict the solution $u$ may be contained in the set
\begin{align}
{\bf W}_{\gamma,\eta}^{\rho}(J\times D) := \Big\{ w\in \mathcal X _{2,\eta} (J\times D): \quad \norm{w(t,.)}\le \rho t^{-\gamma}, \textrm{ for } 0<t\le T  \Big\} ,\nn
\end{align}
for $\rho>0$, $0<\gamma\le \eta <1$.

\vspace*{0.2cm}

The prediction will be proved in the next lemma. However, it is necessary to give some useful notes on ${\bf W}_{\gamma,\eta}^{\rho}(J\times D)$ as follows. For $w\in {\bf W}_{\gamma,\eta}^{\rho}(J\times D)$,  we see
\begin{align}
\esssup_{0\le t\le T}\int_0^t \norm{w(\tau,.)}  (t-\tau)^{\eta-1} d\tau \le   \rho \esssup_{0\le t\le T} \int_0^t \tau^{-\gamma} (t-\tau)^{\eta-1} d\tau. \nn
\end{align}
The function $\tau \to   \tau^{-\gamma} (t-\tau)^{\eta-1}$ is integrable on $(0,t)$ since both numbers $-\gamma$ and $\eta-1$ are greater than $-1$. In addition, we have  $\displaystyle \int_0^t \tau^{-\gamma} (t-\tau)^{\eta-1} d\tau = t^{\eta-\gamma}B(\eta,1-\gamma)$, where $B(\cdot,\cdot)$ is the Beta function see, for example, \cite{Samko, Podlubny, Diethelm}. Hence, we have
\begin{align}
\vertiii{w}_{\mathcal X _{2,\eta}}  \le  \rho T^{\eta-\gamma}B(\eta,1-\gamma). \label{betabound}
\end{align}
Moreover, if $\gamma <\eta$, then there always exists a real number $p$ such that $1<\frac{1}{\eta} < p < \frac{1}{\gamma}$. This implies that the function $t^{-\gamma}$ belongs to $L^p(0,T;L^2(D))$. Therefore, we can obtain the following inclusions
\begin{align}
{\bf W}_{\gamma,\eta}^{\rho}(J\times D) \subset L^p(0,T;L^2(D)) \subset \mathcal X _{2,\eta}(J\times D).\nn
\end{align}
In the following lemma, we will consider the case $\gamma = \eta$, which we will denote by ${\bf W}_{\gamma}^{\rho}(J\times D) := {\bf W}_{\gamma,\eta}^{\rho}(J\times D) $.

\vspace*{0.2cm}

Now, the Sobolev embedding ${\bf V}_{\beta p} \hookrightarrow L^2(D)$ shows that there exists a positive constant $C_D$ depending  on $D,\beta,q$ such that $\norm{v}\le C_D \norm{v}_{{\bf V}_{\beta p}}$ for all $v\in {\bf V}_{\beta p}$. In this section, we let $$k_0(K) =  KB(\alpha q,1-\alpha q)T^{\alpha q} \left[ \widehat{c}_\alpha m_1^{- \beta{p}}   +  \widehat{c}_\alpha^2 c^{-1}_\alpha  (m_1^{-\beta}+ T^\alpha)^{{p}}   \right],$$
and $$M_0 = C_D T^{\alpha q}+\widehat{c}_\alpha c^{-1}_\alpha  T^{\alpha {q}}   (m_1^{-\beta}+T^\alpha)^{{p}}.$$

\begin{lemma} \label{Lemma2} Let $p$, $q$ be defined by (R1). Let $\big\{w_{(n)}\big\}$ be defined by $w_{(0)}=\phi$,
	\begin{align}
	w_{(n)}(t,x)=\mathcal O(t,x)w_{(n-1)}, \quad  n \in \mathbb{N}, n\ge 1, \label{SequenceDefine}
	\end{align}
	where $\mathcal O=\sum_{1\le n\le 3}\mathcal O_n$. If $\phi$ belongs to  ${\bf V}_{\beta{p}}$, $F$ satisfies  (A1), and  $k_0(T)<1$, then
	\begin{align}
	\big\{w_{(n)}\big\}_{n\ge 0} \subset {\bf W}_{\alpha q}^{\widehat{C}_0}(J\times D), \label{SequenceBounded}
	\end{align}
	where $\widehat{C}_0:=  \widetilde{C}_0 \norm{\phi}_{{\bf V}_{\beta{p}}}$ and $\widetilde{C}_0=\frac{M_0}{1- k_0(T) }$.
\end{lemma}

\begin{proof} First, we have $$\norm{w_{(0)}} = \norm{\phi} \le C_D \norm{\phi}_{{\bf V}_{\beta p}} \le M_0  \norm{\phi}_{{\bf V}_{\beta p}} t^{-\alpha q} \le \widehat{C}_0 t^{-\alpha q}.$$ Hence, inequality (\ref{betabound}) and $\alpha q<1$,  imply  $w_{(0)}\in {\bf W}_{\alpha q}^{\widehat{C}_0}(J\times D)$. Now, we assume that $w_{(n-1)}$ belongs to ${\bf W}_{\alpha q}^{\widehat{C}_0}(J\times D)$ for some $n\ge 1$. Then, by using (\ref{betabound}), we have
	\begin{align}
	\hspace*{0.5cm} \vertiii{w_{(n-1)}}_{\mathcal X _{2,\eta}}  \le  \widehat{C}_0 B(\alpha q,1-\alpha q). \label{ConductionHypothesis}
	\end{align}
	By induction, the inclusion (\ref{SequenceBounded}) will be proved by showing that  $w_{(n)}$ belongs to ${\bf W}_{\alpha q}^{\widehat{C}_0}(J\times D)$. Indeed, by using the same arguments as in the proof of (\ref{Theo1O1}), we have
	\begin{align}
	\norm{\mathcal O_1 (t,.) F(w_{(n-1)})}  \le\,&  \widehat{c}_\alpha m_1^{- \beta{p}} \int_0^t \norm{F(\tau,.,u_{(n-1)}(\tau,.))} (t-\tau)^{\alpha{q}-1} d\tau  , \label{Lem2a}\\
	\le\,& K \widehat{c}_\alpha m_1^{- \beta{p}} \vertiii{w_{(n-1)}}_{\mathcal X_{2,\alpha q}} \le M_{26}  \widehat{C}_0 t^{-\alpha q}  , \nn
	\end{align}
	where we have used (\ref{FvKv}),  (\ref{ConductionHypothesis}) and let  $M_{26}= K \widehat{c}_\alpha m_1^{- \beta{p}}  B(\alpha q,1-\alpha q)T^{\alpha q}$.
	On the other hand, the norm $\norm{\mathcal O_2 (t,.) \phi }$ is estimated by (\ref{Theo1O2}), i.e.,
	$$\norm{\mathcal O_{2}(t,x)\phi}   \le M_0   \norm{\phi}_{{\bf V}_{\beta{p}}} t^{-\alpha{q}},$$
	where we note that $$M_2=\widehat{c}_\alpha c^{-1}_\alpha  T^{\alpha {q}}   (m_1^{-\beta}+T^\alpha)^{{p}} \le M_0.$$
	 The norm $\norm{\mathcal O_3 (t,.) F(u)}$ can be estimated in the same way as in the proof of (\ref{Theo1O3}). That is,
	\begin{align}
	\norm{\mathcal O_3 (t,.) F(w_{(n-1)})} \le\,&  M_3  t^{-\alpha{q}} \int_0^T \norm{F(\tau,.,u_{(n-1)}(\tau,.))} (T-\tau)^{\alpha{q}-1} d\tau  \label{Lem2b} \\
	\le\,&  KM_3  t^{-\alpha{q}} \vertiii{w_{(n-1)}}_{\mathcal X_{2,\alpha q}} \le M_{27}   \widehat{C}_0 t^{-\alpha{q}} , \nn
	\end{align}
	where (\ref{FvKv}),  (\ref{ConductionHypothesis}) have been used. Here, we let   $M_{27}=KM_3 B(\alpha q,1-\alpha q)$, and $M_3=\widehat{c}_\alpha^2 c^{-1}_\alpha T^{\alpha{q}} (m_1^{-\beta}+ T^\alpha)^{{p}}.$
	
	\vspace*{0.2cm}
	
	We deduce from the above arguments and $w_{(n)}(t,.)  = \mathcal O(t,.)w_{(n-1)} $ that
	\begin{align}
	\norm{w_{(n)}(t,.)}   \le \sum_{1\le n\le 3} \norm{\mathcal O_n(t,.)w_{(n-1)}} \le k_{0}(T)   \widehat{C}_0 t^{-\alpha q}  + M_0   \norm{\phi}_{{\bf V}_{\beta{p}}} t^{-\alpha{q}}, \label{wnbound}
	\end{align}
	by noting that $k_{0}(T)=  M_{26}+M_{27}$. Since $   \widehat{C}_0=\frac{M_0}{1-k_{0}(T)} \norm{\phi}_{{\bf V}_{\beta{p}}}$, the above inequality implies that $\norm{w_{(n)}(t,.)}    \le \widehat{C}_0$. Therefore, from $\alpha q<1$, we obtain the inclusion (\ref{SequenceBounded}).
\end{proof}

Next, it is necessary to give a definition of mild solutions of FVP (\ref{mainpro1})-(\ref{mainpro3}).

\begin{definition} If a function $u$ belongs to $L^p(0,T;L^q(D))$, for some $p,q\ge 1$, and satisfies Equation (\ref{Nmild}), then $u$ is said to be a mild solution of FVP  (\ref{mainpro1})-(\ref{mainpro3}).
\end{definition}

The following theorems presents  existence, uniqueness, and regularity of a mild solution of FVP (\ref{mainpro1})-(\ref{mainpro3}).

\begin{theorem} \label{STheoN1} \hspace*{20cm}\\
	\noindent a) Let $p,q,r,p',q'$ be defined by (R1),  (R4b).  If $\varphi$ belongs to  ${\bf V}_{\beta{p}}$, $F$ satisfies  (A1), and  $k_0(T)<1$, then FVP (\ref{mainpro1})-(\ref{mainpro3}) has a unique solution   $$u \in    L^{\frac{1}{\alpha q'}-r}(0,T;{\bf V}_{\beta({p}- p' )}) \cap C^{\alpha q}((0,T];L^2(D))  ,$$
	and there exists a positive constant $C_5$ such that
	\begin{align}
	\norm{u}_{L^{\frac{1}{\alpha q'}-r}(0,T;{\bf V}_{\beta({p}- p' )})} + \norm{u}_{C^{\alpha q}((0,T];L^2(D))}  \le \,& C_{5} \norm{\varphi}_{{\bf V}_{\beta{p}}} .  \nn
	\end{align}
	
	\noindent b) Let $p,q,r,p',q'$ be defined by (R1b), (R4c).  If $\varphi$ belongs to  ${\bf V}_{\beta{p}}$, $F$ satisfies  (A1), and  $k_0(T)<1$, then FVP (\ref{mainpro1})-(\ref{mainpro3}) has a unique solution $$u \in    L^{\frac{1}{\alpha q'}-r}(0,T;{\bf V}_{\beta({p}- p' )}) \cap C^{\alpha q}((0,T];L^2(D)) \cap C^{\alpha q}([0,T];{\bf V}_{-\beta q'}),$$
	and there exists a positive constant $C_6$ such that
	\begin{align}
	\norm{u}_{L^{\frac{1}{\alpha q'}-r}(0,T;{\bf V}_{\beta({p}- p' )})} + \norm{u}_{C^{\alpha q}((0,T];L^2(D))} + \vertiii{u}_{C^{\alpha q}([0,T];{\bf V}_{-\beta q'})} \le \,& C_{6} \norm{\varphi}_{{\bf V}_{\beta{p}}} .  \nn
	\end{align}
\end{theorem}

\begin{proof} a) We divide the proof of this part  into the following steps.
	
	\vspace*{0.2cm}
	
	\noindent {\bf Step 1:} We prove the existence and uniqueness of a mild solution. In order to prove the existence of a mild solution of FVP (\ref{mainpro1})-(\ref{mainpro3}), we will construct a convergent sequence in $L^{\frac{1}{\alpha q}-r}(0,T;L^2(D))$ whose limit will be a mild solution of the problem. Here, $r$ is defined by (R2). Let $\big\{w_{(n)}\big\}_{n\ge 0}$ be a sequence defined by Lemma \ref{Lemma2} with respect to $\phi=\varphi\in {\bf V}_{\beta p}$, then  $\big\{w_{(n)}\big\}_{n\ge 0} \subset {\bf W}_{\alpha q}^{\widehat{C}_0}(J\times D)$ where $\widehat{C}_0:= \frac{M_2}{1- k_0(T) } \norm{\varphi}_{{\bf V}_{\beta{p}}}.$ Therefore,
	\begin{align}
	\norm{w_{(n)}(t,.)}\le \widehat{C}_0 t^{-\alpha q}, \quad  0 < t\le T, \nn
	\end{align}
	for all $n\ge 1$. This together with $t^{-\alpha{q}}$ belonging to $L^{\frac{1}{\alpha q}-r}(0,T;\mathbb{R})$ implies that $\big\{w_{(n)}\big\}_{n\ge 0}$ is a bounded sequence in $L^{\frac{1}{\alpha q}-r}(0,T;L^2(D))$. Now, we will show that $\big\{w_{(n)}\big\}_{n\ge 0}$ is convergent by proving that it is also a Cauchy sequence.  For fixed $n\ge 1$ and $k\ge 1$, the definition (\ref{SequenceDefine}) of $\big\{w_{(n)}\big\}_{n\ge 0}$ yields that
	\begin{align}
	w_{(n+k)}(t,x) - w_{(n)}(t,x)
	=\,& \mathcal O_1(t,x) \left[ F(w_{(n-1+k)})-F(w_{(n-1)})\right] \nn\\
	+\,& \mathcal O_3(t,x) \left[ F(w_{(n-1+k)})-F(w_{(n-1)})\right] .   \nn
	\end{align}
	Since $F$ satisfies Lemma \ref{Lemma2}, the latter equation shows that we can apply the same arguments as in Lemma \ref{Lemma2}   with $\phi={\bf 0}$. Hence, one can deduce
	\begin{align}
	\,& \norm{w_{(1+k)}(t,.) - w_{(1)}(t,.)} \nn\\
	\le\,& \norm{\mathcal O_1(t,.) \left[ F(w_{(0+k)})-F(w_{(0)})\right]} + \norm{\mathcal O_3(t,x) \left[ F(w_{(0+k)})-F(w_{(0)})\right]} \nn\\
	\le\,& \widehat{c}_\alpha m_1^{- \beta{p}} \int_0^t \norm{F(\tau,.,u_{(0+k)}(\tau,.))-F(\tau,.,u_{(0)}(\tau,.))} (t-\tau)^{\alpha{q}-1} d\tau \nn\\
	+ \,& M_3  t^{-\alpha{q}} \int_0^T \norm{F(\tau,.,u_{(0+k)}(\tau,.))-F(\tau,.,u_{(0)}(\tau,.))} (T-\tau)^{\alpha{q}-1} d\tau \nn
	\end{align}
	where we combined the estimates (\ref{Lem2a}), (\ref{Lem2b}). From   $\big\{w_{(n)}\big\}_{n\ge 0} \subset {\bf W}_{\alpha q}^{\widehat{C}_0}(J\times D)$, we have  $$\hspace*{-1.15cm} \norm{u_{(0+k)}(\tau,.)- u_{(0)}(\tau,.)} \le  2\widehat{C}_0 t^{-\alpha q}.$$ Thus, by noting the identity  $$\displaystyle \int_0^t  (t-\tau)^{a-1} \tau^{b-1} d\tau = t^{a+b-1}B(a,b), $$ where $a,b>0$ and $B(\cdot,\cdot)$ is the Beta function, we find that
	\begin{align}
	\,& \norm{w_{(1+k)}(t,.) - w_{(1)}(t,.)} \nn\\
	\le\,& \widehat{c}_\alpha m_1^{- \beta{p}}  K   \int_0^t \norm{u_{(0+k)}(\tau,.)- u_{(0)}(\tau,.)} (t-\tau)^{\alpha{q}-1}     d\tau, \nn \\
	+ \,& M_{3}K t^{-\alpha{q}}  \int_0^T \norm{u_{(0+k)}(\tau,.)- u_{(0)}(\tau,.)} (T-\tau)^{\alpha{q}-1}    d\tau \nn\\
	\le\,& \widehat{c}_\alpha m_1^{- \beta{p}}  K(2\widehat{C}_0) \int_0^t \tau^{-\alpha q} (t-\tau)^{\alpha{q}-1}     d\tau +  M_{3}K (2\widehat{C}_0) t^{-\alpha{q}}  \int_0^T \tau^{-\alpha q} (T-\tau)^{\alpha{q}-1}    d\tau \nn\\
	\le\,& \left( \widehat{c}_\alpha m_1^{- \beta{p}}  K T^{\alpha q} B(\alpha q, 1-\alpha q) +  M_{3}K     B(\alpha q, 1-\alpha q) \right) (2\widehat{C}_0) t^{-\alpha{q}} . \hspace*{-0.5cm} \nn
	\end{align}
	From the definition of $k_0(T)$, we conclude that $$ \norm{w_{(1+k)}(t,.) - w_{(1)}(t,.)}  \le k_0(T) (2\widehat{C}_0) t^{-\alpha{q}}. $$   Iterating this method $n$-times shows
	\begin{align}
	\,& \norm{w_{(n+k)}(t,.) - w_{(n)}(t,.)}  \le k^{n}_0(T) (2\widehat{C}_0) t^{-\alpha{q}} .\nn
	\end{align}
	Taking the $L^{\frac{1}{\alpha q}-r}(0,T;\mathbb{R})$-norm of both sides of the above inequality directly implies
	\begin{align}
	\,& \norm{w_{(n+k)} - w_{(n)}}_{L^{\frac{1}{\alpha q}-r}(0,T;L^2(D))}  \le k^{n}_0(T) (2\widehat{C}_0) \norm{t^{-\alpha{q}}}_{L^{\frac{1}{\alpha q}-r}(0,T;\mathbb{R})} .\label{wncauchy}
	\end{align}
	Here, we emphasise that the constants in (\ref{wncauchy}) also do not depend on $(n,k)$. Therefore, by letting $n$ go to infinity, we obtain
	$$\lim\limits_{n,k\to \infty} \norm{w_{(n+k)} - w_{(n)}}_{L^{\frac{1}{\alpha q}-r}(0,T;L^2(D))} = 0,$$ i.e., $\big\{w_{(n)}\big\}_{n\ge 0}$ is a bounded Cauchy sequence in $L^{\frac{1}{\alpha q}-r}(0,T;L^2(D))$. Hence, there exists a function $u$ in $L^{\frac{1}{\alpha q}-r}(0,T;L^2(D))$ such that
	$$u= \lim\limits_{n\to \infty}  w_{(n)}, \quad \textrm{in } L^{\frac{1}{\alpha q}-r}(0,T;L^2(D)),$$
	and $u$ satisfies Equation (\ref{Nmild}), i.e., $u$ is a mild solution of FVP problem (\ref{mainpro1})-(\ref{mainpro3}). Moreover, the boundedness (\ref{wnbound}) of $\big\{w_{(n)}\big\}_{n\ge 0}$ gives
	\begin{align}
	\norm{u(t,.)}    \le  \widetilde{C}_0 \norm{\varphi}_{{\bf V}_{\beta{p}}}  t^{-\alpha q}, \label{uboundW}
	\end{align}
	and so that $$\norm{u}_{L^{\frac{1}{\alpha q}-r}(0,T;L^2(D))}    \le  M_{28} \norm{\varphi}_{{\bf V}_{\beta{p}}}$$ where $M_{28}= \widetilde{C}_0 \norm{t^{-\alpha{q}}}_{L^{\frac{1}{\alpha q}-r}(0,T;\mathbb{R})}$.

	Now, we show the uniqueness of the solution $u$. Assume that $\widetilde{u}$ is   another solution of FVP (\ref{mainpro1})-(\ref{mainpro3}). Then, by applying the same argument as in (\ref{wncauchy}), we also have
	\begin{align}
	\,& \norm{u - \widetilde{u}}_{L^{\frac{1}{\alpha q}-r}(0,T;L^2(D))}  \le k^{n}_0(T) (2\widehat{C}_0) \norm{t^{-\alpha{q}}}_{L^{\frac{1}{\alpha q}-r}(0,T;\mathbb{R})} , \nn
	\end{align}
	for all $n\in \mathbb{N}$, $n\ge 1$. Thus $\norm{u - \widetilde{u}}_{L^{\frac{1}{\alpha q}-r}(0,T;L^2(D))}=0$ by letting $n$ go to infinity. Hence, $u = \widetilde{u}$ in $L^{\frac{1}{\alpha q}-r}(0,T;L^2(D))$.

	\noindent {\bf Step 2:} We prove that $u\in L^{\frac{1}{\alpha q'}-r}(0,T;{\bf V}_{\beta({p}- p' )}).$ This will be proved by using the inequality (\ref{uboundW}). We now apply the  same arguments as in the  proofs of (\ref{Theo2O1}), and (\ref{Theo2O3}) to estimate $\norm{u(t,.)}_{{\bf V}_{\beta({p}- p' )}}$ as follows. First, we have
	\begin{align}
	\,& \norm{\mathcal O_1 (t,.) F(u)}_{{\bf V}_{\beta({p}- p' )}} \nn\\
	\le\,& \widehat{c}_\alpha \int_0^t \left\{\sum_{j=1}^\infty F_j^2(\tau,u(\tau)) m_j^{2\beta({p}- p' )} m_j^{-2\beta{p}} (t-\tau)^{-2\alpha{p}} (t-\tau)^{2\alpha-2} \right\}^{1/2} d\tau \nn\\
	\le\,& \widehat{c}_\alpha m_1^{-\beta p' } \int_0^t \norm{F(\tau,.,u(\tau,.))} (t-\tau)^{\alpha{q}-1} d\tau  \nn\\
	\le\,& \widehat{c}_\alpha m_1^{-\beta p' } K \int_0^t \norm{u(\tau,.)} (t-\tau)^{\alpha{q}-1} d\tau  \nn\\
	\le\,& \widehat{c}_\alpha m_1^{-\beta p' }K \widehat{C}_0   \int_0^t \tau^{-\alpha q} (t-\tau)^{\alpha{q}-1} d\tau \le M_{29} \norm{\varphi}_{{\bf V}_{\beta{p}}}  t^{-\alpha q' } ,  \label{O1Fu}
	\end{align}
	where we let $$M_{29}=\widehat{c}_\alpha m_1^{-\beta p' }K \widetilde{C}_0 B(\alpha q, 1-\alpha q) T^{\alpha q'} .$$  Secondly,
	\begin{align}
	 \norm{\mathcal O_3(t,.)F(u)}_{{\bf V}_{\beta({p}- p' )}}  
	\le\,& M_6 t^{-\alpha q' } \int_0^T  \norm{F(\tau,.,u(\tau,.))} (T-\tau)^{\alpha{q}-1} d\tau \nn\\
	\le\,& M_6 t^{-\alpha q' } K \int_0^T \norm{u(\tau,.)} (T-\tau)^{\alpha{q}-1} d\tau  \nn\\
	\le\,& M_6 t^{-\alpha q' } K \widehat{C}_0   \int_0^t \tau^{-\alpha q} (t-\tau)^{\alpha{q}-1} d\tau\nn\\
	& \le M_{30} \norm{\varphi}_{{\bf V}_{\beta{p}}} t^{-\alpha q' } , \hspace*{1.9cm}  \label{O3Fu}
	\end{align}
	where we let $M_{30}= M_6 K \widetilde{C}_0 B(\alpha q, 1-\alpha q)$. We recall that  $\norm{\mathcal O_2(t,.)\varphi}_{{\bf V}_{\beta({p}- p' )}}$ have been estimated by (\ref{Theo2O2}). According to the above arguments, we arrive at the estimate
	\begin{align}
	\norm{u(t,.)}_{{\bf V}_{\beta({p}- p' )}}
	\le \,& \norm{\mathcal O_1(t,.)F(u)}_{{\bf V}_{\beta({p}- p' )}} + \norm{\mathcal O_2(t,.)\varphi}_{{\bf V}_{\beta({p}- p' )}}  + \norm{\mathcal O_3(t,.)F(u)}_{{\bf V}_{\beta({p}- p' )}} \hspace*{-0.8cm} \nn\\
	\le \,& M_{31} \norm{\varphi}_{{\bf V}_{\beta{p}}} t^{-\alpha q' } , \nn
	\end{align}
	for $$M_{31}=M_{29} + M_5 + M_{30}.$$ By taking the $L^{\frac{1}{\alpha q'}-r}(0,T;\mathbb{R})$-norm, then the latter inequalities can be transformed into the following estimate
	\begin{align}
	\norm{u}_{L^{\frac{1}{\alpha q'}-r}(0,T;{\bf V}_{\beta({p}- p' )})}  \le \,& M_{32} \norm{\varphi}_{{\bf V}_{\beta{p}}}  .\label{trenlau1}
	\end{align}
	where $M_{32}=M_{31}\norm{t^{-\alpha q'}}_{L^{\frac{1}{\alpha q'}-r}(0,T;\mathbb{R})}.$
	\noindent {\bf Step 3:} We prove that $u\in C^{\alpha q}((0,T];L^2(D)).$ Let us consider   $0<t_1<t_2\le T$. By the same arguments as in (\ref{ConFor}), we have
	\begin{align}
	\,& u(t_2,x)-u(t_1,x)\nn\\
	=\,&   \sum_{j=1}^\infty \int_0^{t_1} \int_{t_1-\tau}^{t_2-\tau}  F_j(\tau,u(\tau))\omega^{\alpha -2}E_{\alpha ,\alpha -1}(-m_j^\beta \omega^\alpha ) d\omega d\tau e_j(x) \nn\\
	+\,& \sum_{j=1}^\infty\int_{t_1}^{t_2}F_j(\tau,u(\tau)) \widetilde{E}_{\alpha,\alpha}(-m_j^\beta (t_2-\tau)^\alpha) d\tau e_j(x) \nn\\
	- \,& \mathcal L^{\beta} \sum_{j=1}^\infty \varphi_j   \int_{t_1}^{t_2}  \frac{\widetilde{E}_{\alpha ,\alpha  }(-m_j^\beta  \omega^\alpha )}{E_{\alpha ,1}(-m_j^\beta  T^\alpha )} d\omega e_j(x)  \nn\\
	+ \,& \mathcal L^{\beta} \sum_{j=1}^\infty   \int_0^T  \int_{t_1}^{t_2}   F_j(\tau,u(\tau)) \widetilde{E}_{\alpha,\alpha}(-m_j^\beta (T-\tau)^\alpha)  \frac{\widetilde{E}_{\alpha ,\alpha  }(-m_j^\beta  \omega^\alpha )}{E_{\alpha ,1}(-m_j^\beta  T^\alpha )} d\omega d\tau e_j(x)\nn\\
	:=\,& \mathcal I^N_1+\mathcal I^N_2 + \mathcal I_3  + \mathcal I^N_4. \label{INn}
	\end{align}
	Here,  by (\ref{I3}), $\norm{\mathcal I_3}$ tends to $0$ as $t_2-t_1$ tends to $0$. In what follows, we will establish the convergence for $\norm{\mathcal I^N_n}$, $n=1,2,4$ which can be treated similarly as in (\ref{I1}), (\ref{I2}), (\ref{I4}) based on the Lipschitzian assumption (A1). We first see that
	\begin{align}
	\norm{\mathcal I^N_1}
	\le \,& \widehat{c}_\alpha m_1^{-\beta{p}} \int_0^{t_1} \left\{ \sum_{j=1}^\infty F_j^2(\tau,u(\tau)) \left| \int_{t_1-\tau}^{t_2-\tau}  \omega^{\alpha -2}  \omega^{-\alpha {p}}   d\omega \right|^2 \right\}^{1/2}    d\tau \nn\\
	\le \,& \frac{\widehat{c}_\alpha m_1^{-\beta{p}}}{1-\alpha q}  \int_0^{t_1} \norm{F(\tau,.,u(\tau,.))}  \Big[ (t_1-\tau)^{\alpha q -1} - (t_2-\tau)^{\alpha q-1}\Big]  d\tau \nn \\
	\le \,& \frac{\widehat{c}_\alpha m_1^{-\beta{p}}}{1-\alpha q} K \int_0^{t_1} \norm{u(\tau,.)}  \Big[ (t_1-\tau)^{\alpha q -1} - (t_2-\tau)^{\alpha q-1}\Big]  d\tau \nn \\
	\le \,& \frac{\widehat{c}_\alpha m_1^{-\beta{p}}}{1-\alpha q} K \widetilde{C}_0  \norm{\varphi}_{{\bf V}_{\beta{p}}}   \int_0^{t_1} \tau^{-\alpha q} \Big[ (t_1-\tau)^{\alpha q -1} - (t_2-\tau)^{\alpha q-1}\Big]  d\tau. \label{IN1}
	\end{align}
	Note that $$\displaystyle \int_0^{t_1} \tau^{-\alpha q} (t_1-\tau)^{\alpha q-1} d\tau = B(\alpha q, 1-\alpha q).$$ Thus, due to the substitution $\tau=t_2 \mu $, we have
	\begin{align}
	&\int_0^{t_1} \tau^{-\alpha q} \Big[ (t_1-\tau)^{\alpha q-1} - (t_2-\tau)^{\alpha q-1}\Big] d\tau \nn\\
	&\quad \quad \quad = B(\alpha q, 1-\alpha q) - \int_0^{t_1/t_2} \mu^{-\alpha q} (1-\mu)^{\alpha q-1} d\mu  \nn\\
	&\quad \quad \quad= B(\alpha q, 1-\alpha q) - \left[ B(\alpha q, 1-\alpha q) - \int_{t_1/t_2}^1  \mu^{-\alpha q} (1-\mu)^{\alpha q-1} d\mu \right] \hspace*{1.55cm} \nn\\
	&\quad \quad \quad \le  \left(\frac{t_2}{t_1}\right)^{\alpha q} \int_{t_1/t_2}^1   (1-\mu)^{\alpha q-1} d\mu = \frac{1}{\alpha q} \left(\frac{t_2}{t_1}-1 \right)^{\alpha q}.\label{CalBeta}
	\end{align}
	As a consequence,  $\displaystyle \lim\limits_{t_2-t_1\to 0} \int_0^{t_1} \tau^{-\alpha q} \Big[ (t_1-\tau)^{\alpha q-1} - (t_2-\tau)^{\alpha q-1}\Big] d\tau = 0$, and so
	$\lim\limits_{t_2-t_1\to 0} \norm{\mathcal I^N_1} = 0$.
	Secondly we proceed to deal with $\mathcal I^N_2$. Now  $0<p_0<p$, by (\ref{BasicMLInequality}), we have 
	\begin{equation}
	|E_{\alpha,\alpha}(-m_j^\beta (t_2-\tau)^\alpha)|\le \widehat{c}_\alpha m_1^{-\beta(p-p_0 )} (t_2-\tau)^{-\alpha(p-p_0 )}.
	\end{equation}
	We deduce the following chain of estimates
	\begin{align}
	\norm{\mathcal I^N_2}
	\le\,& \int_{t_1}^{t_2}\norm{\sum_{j=1}^\infty F_j(\tau,u(\tau)) E_{\alpha,\alpha}(-m_j^\beta (t_2-\tau)^\alpha)  e_j} (t_2-\tau)^{\alpha-1} d\tau \nn \\
	\le\,& \widehat{c}_\alpha m_1^{-\beta(p-p_0 )} \int_{t_1}^{t_2}\norm{F(\tau,.,u(\tau,.))} (t_2-\tau)^{\alpha q -1+\alpha p_0 } d\tau \nn \\
	\le\,& \widehat{c}_\alpha m_1^{-\beta(p-p_0 )} K  \int_{t_1}^{t_2}\norm{u(\tau,.)} (t_2-\tau)^{\alpha q -1 } d\tau (t_2-t_1)^{\alpha p_0 } \nn\\
	\le\,& \widehat{c}_\alpha m_1^{-\beta(p-p_0 )} K \widetilde{C}_0  \norm{\varphi}_{{\bf V}_{\beta{p}}} \int_{0}^{t_2} \tau^{\alpha q -1} (t_2-\tau)^{\alpha q -1 } d\tau (t_2-t_1)^{\alpha p_0 } \nn  \\
	\le\,& \widehat{c}_\alpha m_1^{-\beta(p-p_0 )} K \widetilde{C}_0  \norm{\varphi}_{{\bf V}_{\beta{p}}} B(\alpha q, 1- \alpha q) (t_2-t_1)^{\alpha p_0 }.  \label{IN2}
	\end{align}
	This implies $\lim\limits_{t_2-t_1\to 0} \norm{\mathcal I^N_2} = 0.$
	Next, we thirdly proceed to consider $\mathcal I^N_4$. The same argument as in (\ref{I4}) gives
	\begin{align}
	\norm{\mathcal I^N_4}
	\le\,& \widehat{c}_\alpha M_{11} \frac{t_2^{\alpha {q}} - t_1^{\alpha {q}}}{t_1^{2\alpha {q}}} \int_0^T  \norm{F(\tau,.,u(\tau,.))}   (T-\tau)^{\alpha{q}-1}   d\tau  \nn\\
	\le\,& \widehat{c}_\alpha M_{11} K \widetilde{C}_0  \norm{\varphi}_{{\bf V}_{\beta{p}}} \frac{t_2^{\alpha {q}} - t_1^{\alpha {q}}}{t_1^{2\alpha {q}}} \int_0^T  \tau^{-\alpha q}   (T-\tau)^{\alpha{q}-1}   d\tau \hspace*{0.2cm} \nn\\
	\le\,& \widehat{c}_\alpha M_{11} K \widetilde{C}_0 B(\alpha q, 1-\alpha q) \norm{\varphi}_{{\bf V}_{\beta{p}}} \frac{t_2^{\alpha {q}} - t_1^{\alpha {q}}}{t_1^{2\alpha {q}}} , \label{IN4}
	\end{align}
	and we arrive at $\lim\limits_{t_2-t_1\to 0} \norm{\mathcal I^N_4} = 0.$ The above arguments prove $u\in C((0,T];L^2(D))$. This combines with (\ref{uboundW}) so that $u\in C^{\alpha q}((0,T];L^2(D))$, and
	\begin{align}
	\norm{u}_{C^{\alpha q}((0,T];L^2(D))}    \le  \widetilde{C}_0   \norm{\varphi}_{{\bf V}_{\beta{p}}} .\label{trenlau2}
	\end{align}
	We complete Step 1 by combining  the inequalities (\ref{trenlau1}) and (\ref{trenlau2}).
	
	\vspace*{0.2cm}	
	
	\noindent b) According to Part a), we have just to prove that  $u \in  C^{\alpha (p-p')}([0,T];{\bf V}_{-\beta q'}).$  In this part, we consider  $0\le t_1<t_2\le T$.  Let us first show
	\begin{align}
	\norm{\mathcal I^N_1}_{{\bf V}_{-\beta q'}} \le M_{33} \norm{\varphi}_{{\bf V}_{\beta{p}}} (t_2-t_1)^{\alpha q}, \label{IN1closed}
	\end{align}
	for some positive constant $M_{33}$, where the case $t_1=0$ is trivial. It is necessary to prove (\ref{IN1closed}) for $t_1>0$. From the proof of the estimate (\ref{I1aaa}), we have
	\begin{align}
	\norm{\mathcal I^N_1}_{{\bf V}_{-\beta {q}'}}
	\le \,& \int_0^{t_1} \left\{ \sum_{j=1}^\infty m_j^{ -2\beta {q}' } F_j^2(\tau,u(\tau)) \left| \int_{t_1-\tau}^{t_2-\tau}  \omega^{\alpha -2} \left| E_{\alpha ,\alpha -1}(-m_j^\beta \omega^\alpha )\right| d\omega \right|^2 \right\}^{1/2}    d\tau . \nn
	\end{align}
	In addition, the inequalities (\ref{BasicMLInequality}) yield that, $|E_{\alpha,\alpha-1}(-m_j^\beta \omega^\alpha)| \le  \widehat{c}_\alpha m_j^{-\beta p'} \omega^{-\alpha p' }$. This associates with $\alpha q' -2=\alpha q -2 + \alpha (p-p')$ so that
	\begin{align}
	\norm{\mathcal I^N_1}_{{\bf V}_{-\beta {q}'}}
	\le \,& \widehat{c}_\alpha m_1^{ - \beta } \int_0^{t_1} \norm{F(\tau,.,u(\tau,.))}   \left| \int_{t_1-\tau}^{t_2-\tau}  \omega^{\alpha q -2 + \alpha (p-p')} d\omega \right|  d\tau   \nn\\
	\le \,& \frac{\widehat{c}_\alpha m_1^{ - \beta }}{1-\alpha q'}  K  \int_0^{t_1} \norm{ u(\tau,.) }   \left[(t_1-\tau)^{\alpha q -1 + \alpha (p-p')} -  (t_2-\tau)^{\alpha q -1 + \alpha (p-p')} \right]  d\tau   \nn \\
	\le \,& \frac{\widehat{c}_\alpha m_1^{ - \beta }}{1-\alpha q'}  K \widehat{C}_0     \int_0^{t_1} \tau^{-\alpha q} (t_1-\tau)^{\alpha (p-p')}\left[(t_1-\tau)^{\alpha q -1} -  (t_2-\tau)^{\alpha q -1} \right]  d\tau  \nn \\
	\le \,& \frac{\widehat{c}_\alpha m_1^{ - \beta }}{1-\alpha q'}  K \widehat{C}_0    t_1^{\alpha (p-p')}  \frac{1}{\alpha q} \left(\frac{t_2}{t_1}-1 \right)^{\alpha q} \le  M_{33}  \norm{\varphi}_{{\bf V}_{\beta{p}}} (t_2-t_1)^{\alpha q}, \nn
	\end{align}
	where we let $$M_{33}= \frac{\widehat{c}_\alpha m_1^{ - \beta}}{\left(1-\alpha q'\right)\alpha q } K \widetilde{C}_0 T^{\alpha(p-p')-\alpha q}.$$
	 Here, we have used (\ref{ConFor}), (\ref{CalBeta}), and $\alpha(p-p')\ge \alpha q$ by (R4b). Secondly, we are going to consider $\mathcal I^N_2$. The Sobolev embedding $L^2(D) \hookrightarrow {\bf V}_{-\beta {q}'}$ yields that there exists a positive constant $M_{34}$ such that
	\begin{align}
	\norm{\mathcal I^N_2}_{{\bf V}_{-\beta {q}'}}
	\le \,& M_{34} \norm{\mathcal I^N_2}  \nn\\
	\le\,& M_{34} \widehat{c}_\alpha m_1^{-\beta p'} K \widetilde{C}_0   B(\alpha q, 1- \alpha q) \norm{\varphi}_{{\bf V}_{\beta{p}}} (t_2-t_1)^{\alpha (p-p') }    \nn\\
	\le\,& M_{34} \widehat{c}_\alpha m_1^{-\beta p'} T^{\alpha(p-p')-\alpha q} K \widetilde{C}_0   B(\alpha q, 1- \alpha q) \norm{\varphi}_{{\bf V}_{\beta{p}}} (t_2-t_1)^{\alpha q } ,   \nn
	\end{align}
	where we applied (\ref{IN2}) with respect to $0<p_0=p-p'<p$. Thirdly, we now consider $\mathcal I^N_3$. By applying the same arguments as in (\ref{I3COT}), one can get
	\begin{align}
	\norm{\mathcal I^N_3}_{{\bf V}_{-\beta {q}'}}
	\le\,&   \frac{M_{14}}{\alpha({p}-{p}')}    \norm{\varphi}_{{\bf V}_{\beta p}} \left(t_2^{\alpha(p-p')}-t_1^{\alpha(p-p')} \right) \nn\\
	\le\,& \frac{M_{14}}{\alpha({p}-{p}')} T^{\alpha(p-p')-\alpha q} \norm{\varphi}_{{\bf V}_{\beta p}} \left(t_2-t_1\right) ^{\alpha q}   . \hspace*{1cm}
	\end{align}
	Finally, the arguments  proving (\ref{I4COT}) give
	\begin{align}
	\norm{\mathcal I_4}_{{\bf V}_{-\beta {q}'}}
	\le\,& \widehat{c}_\alpha \frac{M_{14}}{\alpha({p}-{p}')}  \left(t_2^{\alpha(p-p')}-t_1^{\alpha(p-p')} \right) \int_0^T  \norm{F(\tau,.,u(\tau,.))}   (T-\tau)^{\alpha{q}-1}   d\tau   \nn\\
	\le\,& \widehat{c}_\alpha \frac{M_{14}}{\alpha({p}-{p}')}  \left(t_2-t_1\right)^{\alpha(p-p')} K \int_0^T  \norm{u(\tau,.)}   (T-\tau)^{\alpha{q}-1}   d\tau  \nn\\
	\le\,& \widehat{c}_\alpha \frac{M_{14}}{\alpha({p}-{p}')}   T^{\alpha(p-p')-\alpha q} K \widetilde{C}_0    B(\alpha q, 1-\alpha q) \norm{\varphi}_{{\bf V}_{\beta{p}}} \left(t_2-t_1\right)^{\alpha(p-p')} .\nn
	\end{align}
	Taking the above estimates for $\mathcal I^N_n$, $1\le n\le 4$ together, we conclude that $u$ belongs to $C^{\alpha q}([0,T];{\bf V}_{-\beta q'})$. Moreover, there exists a positive constant $M_{35}$ such that
	\begin{align}
	\vertiii{u}_{C^{\alpha q}([0,T];{\bf V}_{-\beta q'})} \le M_{35} \norm{\varphi}_{{\bf V}_{\beta{p}}}. \nn
	\end{align}
	By combining this inequality with (\ref{trenlau1}), (\ref{trenlau2}), we complete the proof.
\end{proof}

\begin{theorem} \label{STheoN2} Let $p,q,p',q',\widehat{p},\widehat{q},r,\widehat{r}$ be defined by (R1), (R4b), (R5b).  If $\varphi$ belongs to  ${\bf V}_{\beta({p}+\widehat{q})}$, $F$ satisfies the assumptions  (A2), and  $k_0(T)<1$, then FVP (\ref{mainpro1})-(\ref{mainpro3}) has a unique solution $u$ satisfying that
	\begin{align}
	\,& u \in    L^{\frac{1}{\alpha q'}-r}(0,T;{\bf V}_{\beta({p}- p' )}) \cap C^{\alpha q}((0,T];L^2(D)), \nn \\
	\,& {}^{c}D_t^\alpha u \in L^{\frac{1}{\alpha \widehat{q}}-\widehat{r}}(0,T;{\bf V}_{-\beta(q+\widehat{p}) }) \cap C^\alpha((0,T];{\bf V}_{-\beta {q}}). \nn
	\end{align}
	Moreover, there exists a constant $C_7>0$ such that
	\begin{align}
	\norm{ {}^{c}D_t^\alpha u}_{L^{\frac{1}{\alpha }-\widehat{r}}(0,T;{\bf V}_{-\beta(q-\widehat{q}) })} + \norm{ {}^{c}D_t^\alpha u(t,.)}_{C^\alpha((0,T];{\bf V}_{-\beta q })}  \le C_7 \norm{\varphi}_{{\bf V}_{\beta (p+\widehat{q})}}. \label{kqqqq}
	\end{align}
\end{theorem}

\begin{proof} Since $F$ satisfies (A2), $F$ also satisfies (A1) with respect to the Lipschitz constant $K_*$. In addition, the Sobolev imbedding ${\bf V}_{\beta(p+\widehat{q})} \hookrightarrow {\bf V}_{\beta p}$ shows that $\varphi$ belongs to ${\bf V}_{\beta p}$. Hence, by Theorem \ref{STheoN1}, FVP (\ref{mainpro1})-(\ref{mainpro3}) has a unique solution
	$$u\in L^{\frac{1}{\alpha q'}-r}(0,T;{\bf V}_{\beta({p}- p' )}) \cap C^{\alpha q}((0,T];L^2(D)).$$
	Moreover, the inequality (\ref{uboundW}) also holds. We deduce that, for $0<t\le T$,
	\begin{align}
	\norm{F(t,.,u(t,.))} \le K_*\widetilde{C}_0\norm{\varphi}_{{\bf V}_{\beta p}} t^{-\alpha q} \le M_{36} K_*\widetilde{C}_0\norm{\varphi}_{{\bf V}_{\beta (p+\widehat{q})}} t^{-\alpha q} . \label{FL2bound}
	\end{align}
	The remainder of this proof falls naturally into two steps as follows. \vspace*{0.2cm}

	\noindent {\bf Step 1:} We prove ${}^{c}D_t^\alpha u$ finitely exists and belongs to $L^{\frac{1}{\alpha  }-\widehat{r}}(0,T;{\bf V}_{-\beta(q-\widehat{q}) })$. By the same way as in Part a of Theorem \ref{STheo2}, we have
	\begin{align}
	{}^{c}D_t^\alpha u_j(t)
	=\,& F_j(t,u(t))  -m_j^\beta F_j(t,u(t))\star \widetilde{E}_{\alpha,\alpha}(-m_j^\beta t^\alpha) \nn\\
	-\,&   \varphi_j  \frac{m_j^\beta E_{\alpha,1}(-m_j^\beta t^\alpha)}{E_{\alpha,1}(-m_j^\beta T^\alpha)} + F_j(T,u(T))\star \widetilde{E}_{\alpha,\alpha}(-m_j^\beta T^\alpha) \frac{m_j^\beta E_{\alpha,1}(-m_j^\beta t^\alpha)}{E_{\alpha,1}(-m_j^\beta T^\alpha)} \nn\\
	:=\,& F_j(t,u(t))+\psi_{j}^{N,1}(t)+\psi_j^{N,2}(t)+\psi_j^{N,3}(t).  \nn
	\end{align}
	for all $j\in \mathbb{N}$, $j\ge 1$. In view of (\ref{FL2bound}), $F(t,.,u(t,.))$ is contained in $L^2(D)$ for $0<t\le T$. This associates with the Sobolev embedding $L^2(D) \hookrightarrow {\bf V}_{-\beta(q-\widehat{q})}$ that $F(t,.,u(t,.))$ is contained in ${\bf V}_{-\beta(q-\widehat{q})}$, namely $\sum_{j=1}^\infty F_j(t,u(t)) e_j$ is contained in ${\bf V}_{-\beta(q-\widehat{q})}$. On the other hand, $\psi_j^{N,2}=\psi_j^{(2)}$, and the norm $\norm{\sum_{j=1}^\infty\psi_j^{N,2}(t)e_j}_{{\bf V}_{-\beta(q-\widehat{q})}}$ exists finitely by (\ref{psi2}). Now, we    consider $\norm{\sum_{j=1}^\infty\psi_j^{N,n}(t)e_j}_{{\bf V}_{-\beta(q-\widehat{q})}}$, $n=1,3$. According to the estimates (\ref{psi1}) and (\ref{psi3}), the following ones hold:
	\begin{align}
	\,&\norm{\sum_{n_1\le j\le n_2} \psi^{N,1}_j(t)e_j}_{{\bf V}_{-\beta(q-\widehat{q}) }}  \le \widehat{c}_\alpha  \int_0^t (t-\tau)^{\alpha(q-\widehat{q})-1} \left\{ \sum_{n_1\le j\le n_2}  F_j^2(\tau,u(\tau))   \right\}^{1/2} d\tau,   \nn \\
	\,&	\norm{\sum_{n_1\le j\le n_2} \psi^{N,3}_j(t)e_j}_{{\bf V}_{-\beta(q-\widehat{q}) }}
	\le \widehat{c}_\alpha M_{19} t^{-\alpha}  \int_0^T   (T-\tau)^{\alpha(q-\widehat{q}) - 1} \left\{ \sum_{n_1\le j\le n_2}  F_j^2(\tau,u(\tau)) \right\}^{1/2}  d\tau.  \nn
	\end{align}
	For $0<\tau<T$, we have $F(\tau,.,u(\tau,.))$ belonging to $L^2(D)$. This follows that the sequence $\big\{G_n(\tau) \big\}$, which is defined by  $G_n(\tau)=\left\{ \sum_{j\ge n}  F_j^2(\tau,u(\tau))   \right\}^{1/2}$,     converges pointwise to $0$ as $n$ goes to infinity. Moreover, by (\ref{FL2bound}), we have
	\begin{align}
	\left| (t-\tau)^{\alpha(q-\widehat{q})-1}G_n(\tau) \right|  \le\,&  M_{36} K_*\widetilde{C}_0\norm{\varphi}_{{\bf V}_{\beta (p+\widehat{q})}}   (t-\tau)^{\alpha(q-\widehat{q})-1} \tau^{-\alpha q}. \nn
	\end{align}
	The function $\tau \to   (t-\tau)^{\alpha(q-\widehat{q})-1} \tau^{-\alpha q}$  is integrable on the open interval $(0,t)$, $t>0$, since
	$$
	\int_0^t (t-\tau)^{\alpha(q-\widehat{q})-1} \tau^{-\alpha q} d\tau = t^{-\alpha \widehat{q}} B(\alpha(q-\widehat{q}),1-\alpha q).$$
	Therefore, the dominated convergence theorem yields that
	$$
	\lim\limits_{n \to \infty} \int_0^t (t-\tau)^{\alpha(q-\widehat{q})-1} G_n(\tau) d\tau = 0.$$
	This together with $\left\{ \sum_{n_1\le j\le n_2}  F_j^2(\tau,u(\tau))   \right\}^{1/2} \le G_n(\tau)$ gives
	\begin{align}
	\lim\limits_{n_1,n_2\to \infty} \int_0^t (t-\tau)^{\alpha(q-\widehat{q})-1} \left\{ \sum_{n_1\le j\le n_2}  F_j^2(\tau,u(\tau))   \right\}^{1/2} d\tau = 0.\nn
	\end{align}
	Similarly,  we also have
	\begin{align}
	\lim\limits_{n_1,n_2\to \infty} \int_0^T   (T-\tau)^{\alpha (q-\widehat{q}) - 1} \left\{ \sum_{n_1\le j\le n_2}  F_j^2(\tau,u(\tau)) \right\}^{1/2}  d\tau = 0.\nn
	\end{align}	
	We deduce   $\norm{\sum_{j=1}^\infty\psi_j^{N,n}(t)e_j}_{{\bf V}_{-\beta(q-\widehat{q})}}$, $n=1,3$ exist finitely. Taking all the above arguments together,
	we conclude that  $ \norm{\sum_{j=1}^\infty {}^{c}D_t^\alpha u_j(t)e_j}_{{\bf V}_{-\beta(q-\widehat{q})}} $ finitely exists. In addition, the Sobolev embedding $L^2(D) \hookrightarrow {\bf V}_{-\beta(q-\widehat{q})}$ yields that there exists a positive constant $M_{37}$ such that $$\norm{F(t,.,u(t,.))}_{{\bf V}_{-\beta(q-\widehat{q}) }} \le M_{37}\norm{F(t,.,u(t,.))}.$$ 
	Hence,
	\begin{align}
	\norm{ {}^{c}D_t^\alpha u(t,.)}_{{\bf V}_{-\beta(q-\widehat{q}) }}
	\le\,& \norm{F(t,.,u(t,.))}_{{\bf V}_{-\beta(q-\widehat{q}) }}  + \sum_{1\le n \le 3} \norm{\sum_{j=1}^\infty \psi_{j}^{N,n}(t)e_j}_{{\bf V}_{-\beta(q-\widehat{q}) }}   \nn\\
	\le\,& M_{37}\norm{F(t,.,u(t,.))}   +   \widehat{c}_\alpha  \int_0^t (t-\tau)^{\alpha(q-\widehat{q})-1} \norm{F(\tau,.,u(\tau,.))} d\tau \nn \\
	+\,& M_{19} t^{-\alpha } \norm{\varphi}_{{\bf V}_{\beta(p+\widehat{q})}} +  \widehat{c}_\alpha M_{19} t^{-\alpha }  \int_0^T   (T-\tau)^{\alpha (q-\widehat{q}) - 1} \norm{F(\tau,.,u(\tau,.))}  d\tau. \nn
	\end{align}
	We now note that
	$$    \int_0^t (t-\tau)^{\alpha(q-\widehat{q})-1} \tau^{-\alpha q} d\tau \le  T^{\alpha-\alpha \widehat{q}} B(\alpha(q-\widehat{q}),1-\alpha q) t^{-\alpha},$$ and $$    \int_0^T (T-\tau)^{\alpha(q-\widehat{q})-1} \tau^{-\alpha q} d\tau =  T^{-\alpha \widehat{q}} B(\alpha(q-\widehat{q}),1-\alpha q).$$ This combines with (\ref{FL2bound}) and there exists a constant $M_{38}>0$ such that
	\begin{align}
	\norm{ {}^{c}D_t^\alpha u(t,.)}_{{\bf V}_{-\beta(q-\widehat{q}) }} \le M_{38} \norm{\varphi}_{{\bf V}_{\beta (p+\widehat{q})}} t^{-\alpha}, \label{dauNbounded}
	\end{align}		
	which leads to
	\begin{align}
	\norm{ {}^{c}D_t^\alpha u}_{L^{\frac{1}{\alpha }-\widehat{r}}(0,T;{\bf V}_{-\beta(q-\widehat{q}) })}  \le  M_{38}   \norm{t^{-\alpha }}_{L^{\frac{1}{\alpha  }-\widehat{r}}(0,T;\mathbb{R})} \norm{\varphi}_{{\bf V}_{\beta (p+\widehat{q})}}  .\label{ganxong}
	\end{align}
	
	\vspace*{0.2cm}
	
	\noindent {\bf Step 2:} We prove ${}^{c}D_t^\alpha u \in C^\alpha((0,T];{\bf V}_{-\beta {q}})$.
	We consider $0<t_1<t_2\le T$. A similar argument as in  (\ref{J}) yields
	\begin{align}
	{}^{c}D_t^\alpha u(t_2,x)-{}^{c}D_t^\alpha u(t_1,x) = F(t_2,x,u(t_2,x))-F(t_1,x,u(t_1,x)) + \sum_{1\le n \le 4} \mathcal{J}_n^N,\nn
	\end{align}
	where $\mathcal J_n^N =\mathcal L^\beta \mathcal I_n^N$ and $\mathcal I_n^N$ is defined by (\ref{INn}). By applying the Sobolev embedding $L^2(D) \hookrightarrow {\bf V}_{-\beta {q}}$, there exists a positive constant $M_{39}$ such that
	\begin{align}
	& \lim\limits_{t_2-t_1 \to 0}\norm{F(t_2,.,u(t_2,.))-F(t_1,.,u(t_1,.))}_{{\bf V}_{-\beta {q}}} \nn\\
	&\quad \quad \quad \quad \quad \le \lim\limits_{t_2-t_1 \to 0}M_{39} \norm{F(t_2,.,u(t_2,.))-F(t_1,.,u(t_1,.))}  \nn\\
	&\quad \quad \quad \quad \quad  \le \lim\limits_{t_2-t_1 \to 0}M_{39} K_* \Big( |t_2-t_1| + \norm{u(t_2,.)-u(t_1,.)} \Big) =0, \hspace*{2.1cm} \nn
	\end{align}
	where we note that $u\in C^{\alpha q}((0,T];L^2(D))$. From (\ref{J1}) and (\ref{CalBeta}), we have
	\begin{align}
	\norm{\mathcal J_1^N}_{{\bf V}_{-\beta {q}}} \le \,& \widehat{c}_\alpha \int_0^{t_1} \norm{F(\tau,.,u(\tau,.))} \left| \int_{t_1-\tau}^{t_2-\tau}  \omega^{\alpha q -2}        d\omega \right|     d\tau  \nn\\
	\le \,& \frac{\widehat{c}_\alpha M_{36} K_*\widetilde{C}_0  }{1-\alpha q}  \norm{\varphi}_{{\bf V}_{\beta (p+\widehat{q})}} \int_0^{t_1} \tau^{-\alpha q} \left[ (t_1-\tau)^{\alpha q -1} - (t_2-\tau)^{\alpha q -1} \right]   d\tau  \nn  \\
	\le \,& \frac{\widehat{c}_\alpha M_{36} K_*\widetilde{C}_0  }{(1-\alpha q)\alpha q}  \norm{\varphi}_{{\bf V}_{\beta (p+\widehat{q})}} \left(\frac{t_2}{t_1}-1 \right)^{\alpha q}.\nn
	\end{align}
	In addition, by $$\int_{t_1}^{t_2} \tau^{-\alpha q} (t_2-\tau)^{\alpha q  -1} d\tau = \int_{t_1/t_2}^1 \mu ^{-\alpha q} (1-\mu )^{\alpha q -1} d\mu \le \frac{1}{\alpha q}\left(\frac{t_2}{t_1}-1\right)^{\alpha q}  $$
	 and   (\ref{FL2bound}), we can obtain the following chain of the inequalities
	\begin{align}
	\norm{\mathcal J^N_2}_{{\bf V}_{-\beta {q}}} \le \,& \widehat{c}_\alpha  \int_{t_1}^{t_2} \norm{F(\tau,.,u(\tau,.))} (t_2-\tau)^{\alpha q  -1} d\tau \hspace*{2.1cm} \nn\\
	\le\,&   \widehat{c}_\alpha M_{36} K_*\widetilde{C}_0  \norm{\varphi}_{{\bf V}_{\beta (p+\widehat{q})}} \int_{t_1}^{t_2} \tau^{-\alpha q} (t_2-\tau)^{\alpha q  -1} d\tau  \nn\\
	\le\,& \widehat{c}_\alpha M_{36} K_*\widetilde{C}_0  \norm{\varphi}_{{\bf V}_{\beta (p+\widehat{q})}} \frac{1}{\alpha q} t_1^{-\alpha q} \left(t_2-t_1 \right)^{\alpha q} . 
	\end{align}
	Finally, the norm $\norm{\mathcal J^N_3}_{V_{-\beta {q}}}$ has been estimated by (\ref{J3}), and the norm $\norm{\mathcal I^N_4}_{V_{-\beta {q}}}$
	can be estimated as follows:
	\begin{align}
	\norm{\mathcal J^N_4}_{{\bf V}_{-\beta {q}}}
	\le \,&  \frac{M_{24}}{\alpha} t_1^{-2\alpha} (t_2^\alpha -t_1^\alpha) \int_0^T \norm{F(\tau,.,u(\tau,.))} (T-\tau)^{ \alpha q -1 }  d\tau \nn\\
	\le \,&  \frac{M_{24}}{\alpha} t_1^{-2\alpha} (t_2^\alpha -t_1^\alpha) M_{36} K_*\widetilde{C}_0  \norm{\varphi}_{{\bf V}_{\beta (p+\widehat{q})}}  \int_0^T \tau^{-\alpha q} (T-\tau)^{ \alpha q -1 }  d\tau \hspace*{0.47cm} \nn\\
	\le \,&  \frac{M_{24}}{\alpha} t_1^{-2\alpha} (t_2^\alpha -t_1^\alpha) M_{36} K_*\widetilde{C}_0  \norm{\varphi}_{{\bf V}_{\beta (p+\widehat{q})}} B(\alpha q, 1-\alpha q)   .\nn
	\end{align}
	It follows from the above arguments that ${}^{c}D_t^\alpha u$ belongs to $C((0,T];{\bf V}_{-\beta {q}})$.  On the other hand, the estimate (\ref{dauNbounded}) also holds for $\widehat{p}=0$ and $\widehat{q}=1$, i.e., we have  $$t^{\alpha }\norm{ {}^{c}D_t^\alpha u(t,.)}_{{\bf V}_{-\beta q }}  \le  M_{38}  \norm{\varphi}_{{\bf V}_{\beta p}} ,$$ for $0<t\le T$. Therefore,  there exists a constant $M_{40}>0$ such that
	\begin{align}
	\norm{ {}^{c}D_t^\alpha u(t,.)}_{C^\alpha((0,T];{\bf V}_{-\beta q })}  \le  M_{40}  \norm{\varphi}_{{\bf V}_{\beta (p+\widehat{q})}} .\label{xongroi}
	\end{align}
	by the Sobolev embedding ${\bf V}_{\beta(p+\widehat{q})} \hookrightarrow {\bf V}_{\beta p}$.	The inequality (\ref{kqqqq}) is derived  by taking the inequality (\ref{ganxong}) and (\ref{xongroi}) together. We finally complete the proof.
\end{proof}

\end{document}